\pdfoutput=1
\RequirePackage{ifpdf}
\ifpdf 
\documentclass[pdftex]{sigma}
\else
\documentclass{sigma}
\fi

\usepackage{bbm,upgreek,mathrsfs,yfonts}
\usepackage{genyoungtabtikz}
\usepackage{aurical}

\newtheorem{Theorem}{Theorem}[section]
\newtheorem{Corollary}[Theorem]{Corollary}
\newtheorem{Lemma}[Theorem]{Lemma}
\newtheorem{Proposition}[Theorem]{Proposition}
 { \theoremstyle{definition}
\newtheorem{Definition}[Theorem]{Definition}
\newtheorem{Example}[Theorem]{Example}
\newtheorem{Remark}[Theorem]{Remark} }

\renewcommand{\t}{\mathfrak{t}}

\renewcommand{\a}{\textswab{a}}
\newcommand{\TT}{\mathscr{T}}
\newcommand{\TTS}{\mathscr{T}_\mathrm{sstd}}
\newcommand{\T}{\mathrm{T}}

\newcommand{\sgn}{\mathrm{sgn}}
\newcommand{\Char}{\operatorname{\mathrm{char}}}
\newcommand{\Hom}{\operatorname{Hom}}
\newcommand{\stab}{\operatorname{stab}}
\newcommand{\sym}[1]{\mathfrak{S}_{#1}}
\newcommand{\Ind}{\operatorname{\mathrm{Ind}}}

\newcommand{\F}{\mathbb{F}}
\newcommand{\bb}[1]{\mathbbm{{#1}}}

\newcommand{\R}{\mathscr{R}}
\newcommand{\C}{\mathscr{C}}
\newcommand{\dR}{\text{\Fontauri d}}

\renewcommand{\O}{\mathcal{O}}
\newcommand{\e}{\varepsilon}
\newcommand{\tl}[2]{{#1}^{+#2}} 
\newcommand{\tlsym}[2]{\mathfrak{S}_{#1}^{+#2}}
\newcommand{\tlxi}[2]{\xi_{#1}^{+#2}}
\newcommand{\Z}{\mathbb{Z}}
\newcommand{\Q}{\mathbb{Q}}
\newcommand{\squ}{\mathbin{{\text{\footnotesize $\sqcup$}}}}
\newcommand{\Jhom}{\upvartheta}

\newcommand{\hJhom}{\hat{\Jhom}}
\newcommand{\JHOM}{\Uptheta}

\newcommand{\hJHOM}{{\JHOM}}

\begin{document}

\newcommand{\arXivNumber}{1606.00542}

\renewcommand{\thefootnote}{}

\renewcommand{\PaperNumber}{038}

\FirstPageHeading

\ShortArticleName{Homomorphisms from Specht Modules to Signed Young Permutation Modules}

\ArticleName{Homomorphisms from Specht Modules\\ to Signed Young Permutation Modules\footnote{This paper is a~contribution to the Special Issue on the Representation Theory of the Symmetric Groups and Related Topics. The full collection is available at \href{https://www.emis.de/journals/SIGMA/symmetric-groups-2018.html}{https://www.emis.de/journals/SIGMA/symmetric-groups-2018.html}}}

\Author{Kay Jin LIM~$^\dag$ and Kai Meng TAN~$^\ddag$}

\AuthorNameForHeading{K.J.~Lim and K.M.~Tan}

\Address{$^\dag$~Division of Mathematical Sciences, Nanyang Technological University,\\
\hphantom{$^\dag$}~SPMS-PAP-03-01, 21 Nanyang Link, 637371 Singapore}
\EmailD{\href{mailto:limkj@ntu.edu.sg}{limkj@ntu.edu.sg}}

\Address{$^\ddag$~Department of Mathematics, National University of Singapore,\\
\hphantom{$^\dag$}~Block S17, 10 Lower Kent Ridge Road, 119076 Singapore}
\EmailD{\href{mailto:tankm@nus.edu.sg}{tankm@nus.edu.sg}}

\ArticleDates{Received July 14, 2017, in final form April 18, 2018; Published online April 25, 2018}

\Abstract{We construct a class $\hJHOM_{\R}$ of homomorphisms from a Specht module $S_\Z^{\lambda}$ to a signed permutation module $M_\Z(\alpha|\beta)$ which generalises James's construction of homomorphisms whose codomain is a Young permutation module. We show that any $\phi \in \Hom_{\Z\sym{n}}\big(S_\Z^\lambda, M_\Z(\alpha|\beta)\big)$ lies in the $\mathbb{Q}$-span of $\hJHOM_{\text{sstd}}$, a subset of $\hJHOM_{\R}$ corresponding to semistandard $\lambda$-tableaux of type $(\alpha|\beta)$. We also study the conditions for which $\hJHOM^{\F}_{\mathrm{sstd}}$~-- a~subset of $\Hom_{\F\sym{n}}\big(S_{\F}^\lambda,M_{\F}(\alpha|\beta)\big)$ induced by $\hJHOM_{\mathrm{sstd}}$~-- is linearly independent, and show that it is a~basis for $\Hom_{\F\sym{n}}\big(S_\F^\lambda,M_\F(\alpha|\beta)\big)$ when $\F\sym{n}$ is semisimple.}


\Classification{20C30}

\renewcommand{\thefootnote}{\arabic{footnote}}
\setcounter{footnote}{0}

\section{Introduction}

Modular representation theory of finite groups, unlike its ordinary counterpart, is not at all well understood. Even for ubiquitous groups like the symmetric groups, many fundamental questions remain open. It is therefore important to understand the naturally occurring representations, as these may provide more information about the modular representation theory in general.

Among the naturally occurring representations of the symmetric groups, the Young permutation modules are perhaps the most well-known. James \cite{GJ} studied these modules in detail and exploited their knowledge to obtain important information such as those about the Specht modules.

When the characteristic of the underlying field is not $2$, the signed permutation modules are a natural generalisation of the Young permutation modules. Their indecomposable summands, known as signed Young modules, are first studied by Donkin in~\cite{D}. Subsequently, signed Young modules are shown to be related to irreducible Specht modules. More specifically, Hemmer~\cite{H} showed that irreducible Specht modules are signed Young modules and, recently, Danz and the first author~\cite{DL} described the label explicitly. Signed permutation modules for Iwahori--Hecke algebras of type $A$ are also studied by Du and Rui in~\cite{DR}.

In this paper, we construct a class $\hJHOM_{\R}$ of homomorphisms from the Specht module $S_\Z^{\lambda}$ to the signed permutation module $M_\Z(\alpha|\beta)$, just like James did for the (unsigned) Young permutation module $M^{\mu}_\F$. A subset of this class, denoted $\hJHOM_{\mathrm{sstd}}$, is a direct generalisation of James's semistandard homomorphisms. James's semistandard homomorphisms form a basis for $\Hom_{\F\sym{n}}\big(S_\F^\lambda, M_\F^\mu\big)$, unless $\Char(\F)=2$ and $\lambda$ is $2$-singular. It is therefore natural to seek generalisation of this statement for our $\hJHOM_{\mathrm{sstd}}$. We show that $\Hom_{\Z\sym{n}}\big(S_\Z^\lambda, M_\Z(\alpha|\beta)\big)$ is contained in the $\Q$-span of $\hJHOM_{\mathrm{sstd}}$, and provide a sufficient condition for $\hJHOM_{\mathrm{sstd}}$ to be linearly independent when reduced modulo $p$. In particular, we prove that when $\F\sym{n}$ is semisimple, with $\Char(\F) = p$, then reducing $\hJHOM_{\mathrm{sstd}}$ modulo~$p$ does indeed give a basis for $\Hom_{\F\sym{n}}\big(S_\F^\lambda, M_\F(\alpha|\beta)\big)$.

We found many examples in our study which show that the signed permutation modules behave in a much more unpredictable way than the unsigned ones. We include some of these in Examples~\ref{Eg: James res}, \ref{E:Jameseg}, \ref{Eg: hom char} and~\ref{E:notli}.

The paper is organised as follows. In the next section, we give a quick introduction of the background.
In Section~\ref{S:sstd hom}, we generalise James's construction to obtain homomorphisms from the Specht module $S_\Z^{\lambda}$ to the signed permutation module $M_\Z(\alpha|\beta)$, and in Section~\ref{S:basis?}, we look into the conditions for which a subset of our constructed homomorphisms, corresponding to the semistandard $\lambda$-tableaux of type $(\alpha|\beta)$, will be a basis for $\Hom_{\F\sym{n}}\big(S^{\lambda}_\F, M_\F(\alpha|\beta)\big)$.

\section{Preliminaries}

In this section, we provide the necessary background and introduce some notations that we shall use in this paper.

Throughout, we fix a field $\F$ of arbitrary characteristic.

\subsection{Symmetric groups}

Let $X$ be a finite set. The symmetric group $\sym{X}$ on $X$ is the group of bijections from $X$ to $X$ under composition of functions. By convention, $\sym{\varnothing}$ is the trivial group. When $Y$ is a non-empty subset of $X$, we view $\sym{Y}$ as a subgroup of $\sym{X}$ by identifying an element of $\sym{Y}$ with its extension that sends $x$ to $x$ for all $x \in X \setminus Y$.

Let $X \subseteq \mathbb{Z}^+$, and $k \in \mathbb{Z}^+$. Define the $k$-translated subset $\tl{X}{k}$ of $\mathbb{Z}^+$ by
\begin{gather*}\tl{X}{k} = \{ x+ k\colon x \in X \}. \end{gather*}
Let $\tlsym{X}{k}=\big\{\tl{f}{k}\colon f\in\sym{X}\big\}$ where $\tl{f}{k}\colon \tl{X}{k}\to\tl{X}{k}$ is defined by
\begin{gather*}
\tl{f}{k}(x+k) = f(x) + k
\end{gather*}
for all $x \in X$. Clearly, $f \mapsto \tl{f}{k}$ is a group isomorphism from $\sym{X}$ to $\tlsym{X}{k}$. For any $S \subseteq \sym{X}$, write $\tl{S}{k}$ for $\{ \tl{s}{k}\colon s \in S\}$.

For $n \in \mathbb{Z}^+$, we write $\sym{n}$ for $\sym{\{1,2,\dotsc, n\}}$, the usual symmetric group on $n$ letters.

\subsection{Compositions and Young subgroups}

A composition is a finite sequence of positive integers. Let $\lambda = (\lambda_1,\lambda_2,\dotsc, \lambda_r)$ be a composition. Then $\ell(\lambda) = r$. If $\sum\limits_{i=1}^{\ell(\lambda)} \lambda_i = n$, we say that $\lambda$ is a composition of $n$, and write $|\lambda| = n$. If in addition $\lambda_1 \geq \lambda_2 \geq \dotsb \geq \lambda_{\ell(\lambda)}$, then we call $\lambda$ a partition of $n$. The unique composition (or partition) of 0 is denoted by $\varnothing$, and $\ell(\varnothing)=0$.

To a composition $\lambda$ of $n$, we associate the Young subgroup $\sym{\lambda}$ of $\sym{n}$, where{\samepage
\begin{gather*}
\sym{\lambda} = \sym{\lambda_1} \tlsym{\lambda_2}{s(\lambda)_1} \dotsm \tlsym{\lambda_{\ell(\lambda)}}{s(\lambda)_{\ell(\lambda)-1}} \cong \sym{\lambda_1} \times \sym{\lambda_2} \times \dotsb \times \sym{\lambda_{\ell(\lambda)}}
\end{gather*}
and $s(\lambda)_j = \sum\limits_{i=1}^j \lambda_i$ for all $1 \leq j \leq \ell(\lambda)-1$.}

Let $(\alpha|\beta)$ be a bicomposition of $n$, i.e., $\alpha$ and $\beta$ are compositions and $|\alpha| + |\beta| = n$. One may view $(\alpha|\beta)$ as a composition of $n$ by concatenating $\alpha$ with $\beta$, and so we have the associated Young subgroup
\begin{gather*}
\sym{\alpha|\beta} = \sym{\alpha} \tlsym{\beta}{|\alpha|} \cong \sym{\alpha} \times \sym{\beta}.\end{gather*}

\subsection{Young diagrams}

For a partition $\lambda$, its Young diagram $[\lambda]$ is defined as
\begin{gather*}
[\lambda] = \big\{ (i,j) \in \mathbb{Z}^2 \colon 1\leq i \leq \ell(\lambda),\ 1 \leq j \leq \lambda_i \big\}.
\end{gather*}

The conjugate partition of $\lambda$, denoted $\lambda'$, is the partition whose Young diagram $[\lambda'] = \{ (i,j)\colon (j,i) \in [\lambda] \}$.

\subsection{Tableaux} \label{SS:Tableaux}

Let $\lambda$ be a partition of $n$. Define formally a $\lambda$-tableau $\t$ as a bijective function $\t\colon [\lambda] \to \{1,2,\dotsc,n \}$. We usually view this as a labelling of the elements of $[\lambda]$ by numbers $1,2,\dotsc,n$, such that each number appears exactly once. Denote the set of $\lambda$-tableaux by $\TT(\lambda)$.

Through a fixed $\lambda$-tableau $\t$ we obtain a group isomorphism $\sym{n} \cong \sym{[\lambda]}$ where we identify $\sigma\in\sym{n}$ with $\t^{-1} \circ \sigma \circ \t\in \sym{[\lambda]}$. Let $R_{[\lambda]}$ and $C_{[\lambda]}$ be the row and column stabilizers of $[\lambda]$, i.e.,
\begin{gather*}
R_{[\lambda]}= \{ f \in \sym{[\lambda]}\colon \forall\, (i, j) \in [\lambda],\, \exists\, (i,j')\in [\lambda],\, f(i,j) = (i,j') \}, \\
C_{[\lambda]}= \{ f \in \sym{[\lambda]}\colon \forall\, (i, j) \in [\lambda],\, \exists\, (i',j)\in [\lambda],\, f(i,j) = (i',j) \}.
\end{gather*}
Under the above group isomorphism, $R_{[\lambda]}$ corresponds to the row stabilizer $R_{\t}$ of $\t$, i.e., $R_{\t} = \t \circ R_{[\lambda]} \circ \t^{-1} \subseteq \sym{n}$, and similarly, $C_{[\lambda]}$ corresponds to the column stabilizer $C_{\t}$ of $\t$, i.e., $C_{\t}=\t\circ C_{[\lambda]}\circ \t^{-1}$.

Let $\t^{\lambda}$ denote the \textit{initial $\lambda$-tableau}, defined by $\t^{\lambda}(i,j) = \lambda_1 + \dotsb + \lambda_{i-1} + j$ for all $(i,j) \in [\lambda]$. Observe that $R_{\t^{\lambda}} = \sym{\lambda}$.

Post-composition of $\lambda$-tableaux by elements of the symmetric group $\sym{n}$ gives a well-defined, faithful and transitive left action of $\sym{n}$ on $\TT(\lambda)$, i.e., $\sigma \cdot \t = \sigma \circ \t$ for all $\sigma \in \sym{n}$ and $\t \in \TT(\lambda)$. Observe that \begin{gather*}R_{\sigma \cdot \t} = R_{\sigma \circ \t} = (\sigma \circ \t) \circ R_{[\lambda]} \circ (\sigma \circ \t)^{-1} = \sigma \circ R_{\t} \circ \sigma^{-1} = \sigma R_{\t} \sigma^{-1},\end{gather*} and similarly, $C_{\sigma \cdot \t} = \sigma C_{\t} \sigma^{-1}$. In particular, the row stabilizers of the $\lambda$-tableaux are conjugate subgroups of $\sym{\lambda}$.

A $\lambda$-tableau $\t$ is {\em standard} if it is increasing along each row and down each column, i.e., we have both $\t(i,j) < \t(i,j')$ and $\t(i,j)< \t(i',j)$ for all $(i,j),(i,j'),(i',j) \in [\lambda]$ with $j < j'$ and $i < i'$. Write $\TT_{\rm{std}}(\lambda)$ for the set of all standard $\lambda$-tableaux.

Let $(\alpha|\beta)$ be a bicomposition of $n$ and $\lambda$ be a partition of $n$. We view a $\lambda$-tableau $\T$ of type $(\alpha|\beta)$ as a colouring of the nodes of $[\lambda]$ by the colours $\bb{c}_1,\bb{c}_2,\dotsc, \bb{c}_{\ell(\alpha)},\bb{d}_1,\bb{d}_2,\dotsc,\bb{d}_{\ell(\beta)}$ such that there are exactly $\alpha_i$ nodes of colour $\bb{c}_i$ and $\beta_j$ nodes of colour $\bb{d}_j$ for all $1 \leq i\leq \ell(\alpha)$ and $1\leq j\leq \ell(\beta)$. Formally, $\T$ is a function \begin{gather*}\T\colon \ [\lambda]\to \{\bb{c}_1,\bb{c}_2,\dotsc, \bb{c}_{\ell(\alpha)},\bb{d}_1,\bb{d}_2,\dotsc,\bb{d}_{\ell(\beta)} \}\end{gather*} such that $|\T^{-1}(\{\bb{c}_i\})|=\alpha_i$ and $|\T^{-1}(\{\bb{d}_j\})|=\beta_j$ for all $1 \leq i\leq \ell(\alpha)$ and $1\leq j\leq \ell(\beta)$. We write $\TT(\lambda, (\alpha|\beta))$ for the set of all $\lambda$-tableaux of type $(\alpha|\beta)$.

Let $\t_0 \in \TT(\lambda)$ be a fixed $\lambda$-tableau (which may or may not be the initial $\lambda$-tableau $\t^{\lambda}$ defined above). We define the canonical $\lambda$-tableau of type $(\alpha|\beta)$ associated to $\t_0$, denoted as $\T_0$, as follows
\begin{gather*}
\T_0(i,j) = \begin{cases}\bb{c}_k&\text{if $\displaystyle{\sum^{k-1}_{a=1} \alpha_a < \t_0(i,j) \leq \sum^{k}_{a=1} \alpha_a}$ for some $k$,}\\ \bb{d}_k&\text{if $\displaystyle{|\alpha|+\sum^{k-1}_{b=1} \beta_b < \t_0(i,j) \leq |\alpha|+\sum^{k}_{b=1} \beta_b}$ for some $k$,}\end{cases} \end{gather*} namely, the nodes labelled $1,\ldots,\alpha_1$ by $\t_0$ are coloured $\bb{c}_1$, those labelled $\alpha_1 + 1,\dotsc, \alpha_1 + \alpha_2$ are coloured $\bb{c}_2$, and so on.

We also have a transitive left action on $\TT(\lambda, (\alpha|\beta))$ by $\sym{n}$ through $\t_0$ as follows: $\sigma \cdot \T = \T \circ (\t_0)^{-1} \circ \sigma^{-1} \circ \t_0$ for all $\sigma \in \sym{n}$ and $\T \in \TT(\lambda,(\alpha|\beta))$. Thus, if $(i,j) \in [\lambda]$ is labelled $a$ by $\t_0$, while the node labelled $\sigma^{-1}(a)$ by $\t_0$ is coloured $\bb{c}_k$ (respectively, $\bb{d}_k$) by $\T$, then $\sigma \cdot \T$ colours $(i,j)$ with $\bb{c}_k$ (respectively, $\bb{d}_k$). Observe that under this action, the stabiliser of $\T_0$ is $\sym{\alpha|\beta}$. Consequently, $\TT(\lambda, (\alpha|\beta))$ is in a bijective correspondence with the set of left cosets of $\sym{\alpha|\beta}$ in~$\sym{n}$ via $d \cdot \T_0 \leftrightarrow d \sym{\alpha|\beta}$. We set \begin{gather*}\T_d=d\cdot \T_0\end{gather*} so that $\T_d=\T_{d'}$ if and only if $d^{-1}d'\in\sym{\alpha|\beta}$.

We order the colours $\bb{c}_1,\bb{c}_2,\dotsc,\bb{d}_1,\bb{d}_2,\dotsc$ as follows
\begin{gather*}\bb{c}_1<\bb{c}_2<\cdots<\bb{d}_1<\bb{d}_2<\cdots.\end{gather*} A $\lambda$-tableau $\T$ of type $(\alpha|\beta)$ is {\em row semistandard} (respectively, {\em column semistandard}) if every row (respectively, column) of $\T$ is weakly increasing. More precisely, $\T$ is row semistandard if $\T(i,j) \leq \T(i, j')$ for all $(i,j), (i,j') \in [\lambda]$ with $j < j'$, and $\T$ is column semistandard if $\T(i,j) \leq \T(i', j)$ for all $(i,j), (i',j) \in [\lambda]$ with $i < i'$. Following \cite[Section~1.2]{Se} (see also \cite[Section~4]{DR}), we say that
a $\lambda$-tableau $\T$ of type $(\alpha|\beta)$ is {\em semistandard} if and only if
\begin{enumerate}\itemsep=0pt
\item [(i)] $\T$ is both row and column semistandard;
\item [(ii)] $\T(i,j) = \T(i,j')$ for some $(i,j),(i,j') \in [\lambda]$ with $j\ne j'$ only if $\T(i,j) = \bb{c}_k$ for some $k$;
\item [(iii)] $\T(i,j) = \T(i',j)$ for some $(i,j),(i',j) \in [\lambda]$ with $i\ne i'$ only if $\T(i,j) = \bb{d}_k$ for some $k$.
\end{enumerate}
We denote the set of semistandard $\lambda$-tableaux of type $(\alpha|\beta)$ by $\TTS(\lambda, (\alpha|\beta))$.


\begin{Example} \label{E:ss} \hfill
\begin{enumerate}\itemsep=0pt
\item[(i)]
Let $\lambda = (2,2,1)$ and $(\alpha|\beta) = ((2)|(2,1))$. Consider the following four $\lambda$-tableaux of type $(\alpha|\beta)$:
\begin{gather*}
\begin{matrix}
\T_1 \\[4pt]
\young(<\:\bb{c}_1\:><\:\bb{c}_1\:>,<\:\bb{d}_1\:><\:\bb{d}_2\:>,<\:\bb{d}_1\:>)
\end{matrix} \qquad
\begin{matrix}
\T_2 \\[4pt]
\young(<\:\bb{c}_1\:><\:\bb{d}_1\:>,<\:\bb{c}_1\:><\:\bb{d}_2\:>,<\:\bb{d}_1\:>)
\end{matrix} \qquad
\begin{matrix}
\T_3 \\[4pt]
\young(<\:\bb{c}_1\:><\:\bb{d}_1\:>,<\:\bb{c}_1\:><\:\bb{d}_1\:>,<\:\bb{d}_2\:>)
\end{matrix} \qquad
\begin{matrix}
\T_4 \\[4pt]
\young(<\:\bb{c}_1\:><\:\bb{c}_1\:>,<\:\bb{d}_1\:><\:\bb{d}_1\:>,<\:\bb{d}_2\:>)
\end{matrix}
\end{gather*}
Only $\T_1$ is semistandard.

\item[(ii)]
Let $p\geq 2$, $\lambda = \big(2,1^{p+2}\big)$ and $(\alpha|\beta) = \big(\varnothing|\big(p,2^2\big)\big)$. Consider the following $\lambda$-tableaux of type $(\alpha|\beta)$:
\begin{gather*}
\begin{matrix}
\T_1 \\[4pt]
\gyoung(<\:\bb{d}_1\:><\:\bb{d}_1\:>,|2\vdts,<\:\bb{d}_1\:>,<\:\bb{d}_2\:>,<\:\bb{d}_2\:>,<\:\bb{d}_3\:>,<\:\bb{d}_3\:>)
\end{matrix} \qquad
\begin{matrix}
\T_2 \\[4pt]
\gyoung(<\:\bb{d}_1\:><\:\bb{d}_2\:>,|3\vdts,<\:\bb{d}_1\:>,<\:\bb{d}_2\:>,<\:\bb{d}_3\:>,<\:\bb{d}_3\:>) \end{matrix} \qquad
\begin{matrix}
\T_3 \\[4pt]
\gyoung(<\:\bb{d}_1\:><\:\bb{d}_3\:>,|3\vdts,<\:\bb{d}_1\:>,<\:\bb{d}_2\:>,<\:\bb{d}_2\:>,<\:\bb{d}_3\:>) \end{matrix}
\end{gather*}
Then $\T_1$ is not semistandard but $\T_2$ and $\T_3$ are. In fact, $\T_2$ and $\T_3$ are the only semistandard $\lambda$-tableaux of type $(\alpha|\beta)$, i.e., $\TTS(\lambda,(\alpha|\beta))=\{\T_2,\T_3\}$.
\end{enumerate}
\end{Example}

It is easy to see that the above discussion generalises the notion of $\lambda$-tableaux of type $\mu$ in the classical case, in the sense that a (semistandard) $\lambda$-tableau of type~$\mu$ is a~(semistandard) $\lambda$-tableau of type $(\mu|\varnothing)$. We denote by $\TT(\lambda,\mu)=\TT(\lambda,(\mu|\varnothing))$ the set of $\lambda$-tableaux of type~$\mu$ and by $\TTS(\lambda,\mu)=\TTS(\lambda,(\mu|\varnothing))$ the set of semistandard $\lambda$-tableaux of type~$\mu$. By convention, $|\TT(\varnothing,(\varnothing|\varnothing))|=1$.

\subsection{One-dimensional representations of symmetric groups}

For the remainder of this section, let $\O$ be either $\F$ or $\Z$. The signature representation of the symmetric group $\sym{n}$ over $\O$ is denoted by $\sgn \colon \sym{n} \to \{ \pm 1 \} \subseteq \O$. We shall abuse notation and also write $\sgn$ for the $\O\sym{n}$-module associated to it. In this latter context, $\sgn$ is $\O$-free of rank~$1$, with basis $\{\epsilon\}$. We denote the trivial $\O\sym{n}$-module that is $\O$-free of rank~$1$ as $\O$, with basis~$\{\bb{1}\}$. Thus, $\sigma \cdot \bb{1} = \bb{1}$ and $\sigma \cdot \epsilon = \sgn(\sigma) \epsilon$ for all $\sigma \in \sym{n}$.

When $H$ is a subgroup of $\sym{n}$, we write the respective restricted $\O H$-modules as $\O_{H}$ and $\sgn_H$. Sometimes, we shall abuse notation and write them as $\O$ and $\sgn$ when there is no confusion.

\subsection{Young permutation modules and Specht modules}

Let $\mu$ be a composition of $n$. The Young permutation module $M^{\mu}_{\O}$ is the permutation module associated to the left regular action of $\sym{n}$ on the left cosets of $\sym{\mu}$ in $\sym{n}$. In other words, $M^{\mu}_{\O} = \bigoplus_{d \sym{\mu} \in \sym{n}/\sym{\mu}} \O (d\sym{\mu})$, and $\sigma \cdot (d \sym{\mu}) = (\sigma d )\sym{\mu}$ for all $\sigma \in \sym{n}$ and $d\sym{\mu} \in \sym{n}/\sym{\mu}$.

When $\tilde{\mu}$ is a composition of $n$ obtained by rearranging some parts of $\mu$, the Young sub\-groups~$\sym{\mu}$ and~$\sym{\tilde{\mu}}$ are conjugate in $\sym{n}$, so that $M^{\mu}_{\O} \cong M^{\tilde{\mu}}_{\O}$ as $\O\sym{n}$-modules.

Let $\lambda$ be a partition. Given a $\lambda$-tableau $\t$, let $d_{\t} = \t \circ \big(\t^{\lambda}\big)^{-1}$. Then $d_{\t} \in \sym{n}$ and $d_{\t} \cdot \t^{\lambda} = d_{\t} \circ \t^{\lambda} = \t$. We define the polytabloid \begin{gather*}e_{\t}:=\sum_{\sigma \in C_{\t}} \sgn(\sigma) (\sigma \cdot (d_{\t} \sym{\lambda})) \in M^{\lambda}_{\O}.\end{gather*} The Specht module $S^{\lambda}_{\O}$ is defined to be the $\O$-submodule of $M^{\lambda}_{\O}$ spanned by $ \{ e_{\t} \colon \t \in \TT(\lambda)\}$. It is not difficult to see that $\tau \cdot e_{\t} = e_{\tau \cdot \t}$ for $\tau \in \sym{n}$ and $\t \in \TT(\lambda)$, so that $S^{\lambda}_{\O}$ is an $\O \sym{n}$-submodule of $M^{\lambda}_{\O}$.

Let $\O=\F$ and $M^*$ denote the contragradient dual of an $\F \sym{n}$-module $M$. The following isomorphism is well known (see \cite[Theorem~8.15]{GJ}):
\begin{gather*}S^\lambda_\F\otimes \sgn\cong \big(S^{\lambda'}_\F\big)^*.\end{gather*}
Furthermore, $\F\sym{n}$ is semisimple as an algebra if and only if either $\Char(\F)=0$ or $\Char(\F)>n$, in which case, the Specht modules $S^{\lambda}_\F$, as $\lambda$ runs over all the partitions of $n$, give a complete list of pairwise non-isomorphic irreducible $\F\sym{n}$-modules.

\subsection{Signed permutation modules} \label{S:signedpermutation}

Let $(\alpha|\beta)$ be a bicomposition of $n$. We define the signed permutation module $M_{\O}(\alpha|\beta)$ to be
\begin{gather*}
M_\O(\alpha|\beta) = \Ind_{\sym{\alpha} \times \sym{\beta}}^{\sym{n}} (\O \boxtimes \sgn),
\end{gather*}
where $\sym{\alpha} \times \sym{\beta}$ is identified with the Young subgroup $\sym{\alpha|\beta}$ of $\sym{n}$. Thus, when $\Gamma$ is a left transversal of $\sym{\alpha|\beta}$ in $\sym{n}$, $M_\O(\alpha|\beta)$ has a basis $\{ d \otimes \bb{1} \otimes \epsilon \colon d \in \Gamma \}$, and $\sigma \cdot (d \otimes \bb{1} \otimes \epsilon) = \sgn(\xi_{\beta}) (d' \otimes \bb{1} \otimes \epsilon)$ if $\sigma d = d' \xi_{\alpha} \tlxi{\beta}{|\alpha|}$ where $d' \in \Gamma$, $(\xi_{\alpha}, \xi_{\beta}) \in \sym{\alpha} \times \sym{\beta}$.

Observe that if $\tilde{\alpha}$ and $\tilde{\beta}$ are compositions obtained by rearranging some parts of $\alpha$ and $\beta$ respectively, then $M_\O(\tilde{\alpha}|\tilde{\beta}) \cong M_\O(\alpha|\beta)$. Also,
\begin{gather*}
M_\O(\alpha|\beta)\otimes \sgn \cong \Ind_{\sym{\alpha}\times \sym{\beta}}^{\sym{n}}((\O\otimes\sgn)\boxtimes (\sgn\otimes \sgn))\cong M_\O(\beta|\alpha).
\end{gather*}

The Young permutation module $M^{\mu}_\O$, where $\mu$ is a composition of $n$, in the previous subsection is isomorphic to $\Ind_{\sym{\mu}}^{\sym{n}} \O$. Thus, $M^{\mu}_\O \cong M_\O(\mu|\varnothing)$. As such, signed permutation modules generalise Young permutation modules.

\begin{Theorem}[signed Young's rule]\label{T:Specht factors of sgn permutation}
Let $(\alpha|\beta)$ be a bicomposition of $n$. The signed permutation module $M_\F(\alpha|\beta)$ has a Specht filtration in which, for a partition $\lambda$ of $n$, the multiplicity of $S_\F^\lambda$ as a factor of the filtration equals $|\TTS(\lambda, (\alpha|\beta))|$.

Dually, the signed permutation module $M_\F(\alpha|\beta)$ has a dual Specht filtration in which, for a~partition $\lambda$ of $n$, the multiplicity of $S_{\lambda,\F}\cong \big(S^\lambda_\F\big)^*$ as a factor of the filtration equals \linebreak $|\TTS(\lambda, (\alpha|\beta))|$.
\end{Theorem}

\begin{proof}This is essentially proved by Du and Rui in \cite{DR}; we review their proof here. In the proof of \cite[Proposition 5.2]{DR}, they showed that there is a Specht filtration for the signed $q$-permutation module of the Iwahori--Hecke algebra of type $A$ over the ring $\mathbb{Z}[v,v^{-1}]$ (where $v$ is an indeterminate and $q = v^2$) with the correct multiplicity for each Specht factor. As such, under specialization to any field $\F$, we get a Specht filtration for the signed $q$-permutation module with the correct multiplicity for each Specht factor. Since the symmetric group algebra is an Iwahori--Hecke algebra of type $A$ with $v=1=q$, the result follows.
\end{proof}

\begin{Remark}
In a private communication with the authors, Andrew Mathas constructed explicitly a Specht filtration and a dual Specht filtration of the signed permutation module, with the correct multiplicity for each Specht factor, for the Iwahori--Hecke algebra of type $A$.
\end{Remark}

We record some immediate corollaries to the signed Young's rule.

\begin{Corollary} Let $(\alpha|\beta)$ be a bicomposition of $n$.
\begin{enumerate}\itemsep=0pt
\item [$(i)$] $[\sym{n} \colon \sym{\alpha|\beta}] = \sum_{\lambda} |\TT_{\mathrm{std}}(\lambda)| |\TTS(\lambda,(\alpha|\beta))|$, where the sum runs over all partitions $\lambda$ of $n$.
\item [$(ii)$] For any partition $\lambda$ of $n$, $|\TTS(\lambda, (\alpha|\beta))|=|\TTS(\lambda',(\beta|\alpha))|$.
\item [$(iii)$] Let $\tilde{\alpha}$, $\tilde{\beta}$ be compositions obtained by rearranging some parts of $\alpha$ and $\beta$ respectively. For any partition $\lambda$ of $n$, $|\TTS(\lambda, (\tilde{\alpha}|\tilde{\beta}))|=|\TTS(\lambda, (\alpha|\beta))|$.
\item [$(iv)$] For any partition $\lambda$ of $m+n$, $|\TTS(\lambda, (\alpha \squ (1^m)|\beta))|=|\TTS(\lambda, (\alpha|\beta \squ (1^m)))|$ where $\squ$ denotes the concatenation of compositions.
\end{enumerate}
\end{Corollary}

\begin{proof} (i) It is well known that the Specht module $S_\F^{\lambda}$ has dimension $|\TT_{\mathrm{std}}(\lambda)|$ (see, for example, \cite[Theorem~1.1]{Peel}). Taking the dimension of $M_\F(\alpha|\beta)$ and applying Theorem \ref{T:Specht factors of sgn permutation} thus yield part~(i).

(ii) When $\F = \mathbb{Q}$, the Specht modules are the irreducible modules, so that the Specht filtration in Theorem \ref{T:Specht factors of sgn permutation} is in fact a composition series. Thus
\begin{gather*}
|\TTS(\lambda, (\alpha|\beta))|= \big[M_{\mathbb{Q}}(\alpha|\beta):S_{\mathbb{Q}}^\lambda\big] = \big[M_{\mathbb{Q}}(\alpha|\beta)\otimes\sgn:S_{\mathbb{Q}}^\lambda\otimes\sgn\big]\\
\hphantom{|\TTS(\lambda, (\alpha|\beta))|= \big[M_{\mathbb{Q}}(\alpha|\beta):S_{\mathbb{Q}}^\lambda\big]}{} =\big[M_{\mathbb{Q}}(\beta|\alpha):S_\mathbb{Q}^{\lambda'}\big]=|\TTS(\lambda',(\beta|\alpha))|;
\end{gather*} here, $[M:S]$ denotes the composition multiplicity of an irreducible module $S$ in $M$.

(iii, iv) These are proved in a manner similar to but easier than part (ii), using the isomorphisms $M_\Q(\alpha|\beta)\cong M_\Q(\tilde{\alpha}|\tilde{\beta})$ and $M_\Q(\alpha\squ(1^m)|\beta)\cong M_\Q(\alpha|\beta\squ(1^m))$ respectively.
\end{proof}

\section{Homomorphisms}\label{S:sstd hom}

Let $\lambda$ be a partition of $n$ and $\mu$ be a composition of $n$. In \cite[Section~13]{GJ}, James constructed for each $\T \in \TT(\lambda,\mu)$ an $\F\sym{n}$-module homomorphism $\theta_\T \colon M_\F^{\lambda} \to M_\F^{\mu}$, and showed that $\{ \hat{\theta}_{\T} := \theta_{\T}|_{S_\F^{\lambda}}\colon \T \in \TTS(\lambda,\mu) \}$ is a basis for the space $\Hom_{\F\sym{n}}\big(S_\F^{\lambda},M_\F^{\mu}\big)$, unless $\Char(\F)=2$ and $\lambda$ is $2$-singular. In particular, this shows that every homomorphism in $\Hom_{\F\sym{n}}\big(S_\F^{\lambda},M_\F^{\mu}\big)$ is the restriction of a~homomorphism in $\Hom_{\F\sym{n}}\big(M_\F^{\lambda},M_\F^{\mu}\big)$ when $\Char(\F)\neq 2$.

In this section, we shall generalise these homomorphisms to obtain homomorphisms between~Specht modules and signed permutation modules. The next example shows that not every homomor\-phism in $\Hom_{\F\sym{n}}\big(S_\F^{\lambda},M_\F(\alpha|\beta)\big)$ is the restriction of a homomorphism in \linebreak $\Hom_{\F\sym{n}}\big(M_\F^{\lambda},M_\F(\alpha|\beta)\big)$, illustrating the difficulty of such generalisation.

\begin{Example}\label{Eg: James res}
Let $n\in \mathbb{Z}^+$ and let $\Char(\F)=p$ with $0 < p \leq n$. Let $\lambda=(1^n)$ and $(\alpha|\beta)=(\varnothing|(n))$. Then $M_\F^\lambda$ is the regular $\F\sym{n}$-module $\F\sym{n}$, while both $M_\F(\alpha|\beta)=M_\F(\varnothing|(n))$ and $S_\F^\lambda$ are isomorphic to the signature representation $\sgn$ of $\F\sym{n}$. Thus $\Hom_{\F\sym{n}}\big(S_\F^\lambda,M_\F(\alpha|\beta)\big)$ has dimension one.

On the other hand, $\Hom_{\F\sym{n}}\big(M_\F^\lambda,M_\F(\alpha|\beta)\big) \cong \Hom_{\F\sym{n}}(\F\sym{n},\sgn)$ has dimension $1$ with a~basis $\{\theta\colon \F\sym{n} \to \sgn \}$, where $\theta(1_{\sym{n}})=\epsilon$. For any $\sigma\in \sym{n}$, we have
\begin{gather*}
\theta(\sigma)=\sigma \cdot \theta(1_{\sym{n}})=\sgn(\sigma) \epsilon.
\end{gather*}
As a submodule of $M_\F^{\lambda} = \F\sym{n}$, the Specht module $S_\F^{\lambda}$ ($\cong \sgn$) is generated by $\sum\limits_{\sigma \in \sym{n}} \sgn(\sigma) \sigma$. We have
\begin{gather*}
\theta\left (\sum_{\sigma\in\sym{n}}\sgn(\sigma)\sigma\right )=\sum_{\sigma\in\sym{n}}\sgn(\sigma)\theta(\sigma)=\sum_{\sigma\in\sym{n}} \epsilon = (n!) \epsilon = 0,
\end{gather*}
since $p \leq n$. Thus $\theta|_{S_\F^\lambda} = 0$.

This example shows that the map $\Hom_{\F\sym{n}} \big(M_\F^{\lambda}, M_\F(\alpha|\beta)\big) \to \Hom_{\F\sym{n}} \big(S_\F^{\lambda}, M_\F(\alpha|\beta)\big)$ defined by $\theta \mapsto \theta|_{S_\F^{\lambda}}$ is not surjective in general.
\end{Example}

As such, to generalise the homomorphisms constructed by James for signed permutation modules, we should not attempt to generalise $\theta_\T$ and take its restriction to $S_\F^{\lambda}$, but have to generalise $\hat{\theta}_\T$ directly instead. In other words, we need to understand $\hat{\theta}_\T(e_\t)$.

\subsection{James's construction}
Let $\lambda$ be a partition of $n$ and $\mu$ be a composition of $n$. Let $\theta_{\T} \colon M_\F^{\lambda} \to M_\F^{\mu}$ be the $\F\sym{n}$-module homomorphism defined in \cite[Section~13]{GJ}. In this subsection, we study how $\theta_{\T}$ acts on the polytabloids in $M_\F^{\lambda}$.

Fix a $\lambda$-tableau $\t_0$, so that $\sym{n}$ acts on $\TT(\lambda,\mu)$ via $\t_0$, as described in Section~\ref{SS:Tableaux}.
Let $\T_0$ be the canonical $\lambda$-tableau of type $\mu$ associated to $\t_0$; recall that $\stab_{\sym{n}}(\T_0) = \sym{\mu}$.
Then each left coset of $\sym{\mu}$ in $\sym{n}$ corresponds to a $\lambda$-tableau of type $\mu$; let $d_\T\sym{\mu}$ be the left coset corresponding to $\T$.
Recall also the initial $\lambda$-tableau $\t^{\lambda}$, and let $d_{\t_0} = \t_0 \circ (\t^{\lambda})^{-1} \in \sym{n}$.
James defined $\theta_\T \in \Hom_{\F\sym{n}}\big(M_\F^{\lambda}, M_\F^{\mu}\big)$ so that
\begin{gather*}\theta_{\T}(d_{\t_0}\sym{\lambda}) = \sum_{\substack{d\sym{\mu} \in \sym{n}/\sym{\mu}: \\ d\sym{\mu} \subseteq R_{\t_0} d_\T \sym{\mu}}} d \sym{\mu}.\end{gather*}

Let $\t \in \TT(\lambda)$, and let $\rho_{\t} = \t \circ (\t_0)^{-1}$. Then $\rho_{\t} \cdot \t_0 = \rho_{\t} \circ \t_0 = \t$, so that $e_{\t} = e_{\rho_{\t} \cdot \t_0} = \rho_{\t} \cdot e_{\t_0}$. Thus
\begin{align*}\allowdisplaybreaks
\theta_{\T}(e_{\t}) &= \theta_{\T}(\rho_{\t} \cdot e_{\t_0})
= \theta_{\T}\left (\rho_{\t} \cdot \sum_{\sigma \in C_{\t_0}} \sgn(\sigma) (\sigma \cdot (d_{\t_0} \sym{\lambda}))\right ) \\
&= \sum_{\sigma \in C_{\t_0}} \sgn(\sigma) (\rho_{\t} \sigma \cdot \theta_{\T}(d_{\t_0} \sym{\lambda}))
= \sum_{\sigma \in C_{\t_0}} \sgn(\sigma) \left(\rho_{\t}\sigma \cdot \sum_{\substack{d\sym{\mu} \in \sym{n}/\sym{\mu} : \\ d\sym{\mu} \subseteq R_{\t_0} d_\T \sym{\mu}}} d \sym{\mu} \right) \\
&= \sum_{\sigma \in C_{\t_0}} \sgn(\sigma) \sum_{\substack{d\sym{\mu} \in \sym{n}/\sym{\mu} : \\ d\sym{\mu} \subseteq \rho_{\t} \sigma R_{\t_0} d_\T \sym{\mu}}} d \sym{\mu}
= \sum_{d \sym{\mu} \in \sym{n}/\sym{\mu}} a_{\rho_\t^{-1}d,d_\T} \, d\sym{\mu},
\end{align*}where
\begin{gather*}
a_{\rho_\t^{-1}d,d_\T}= \sum_{\substack{\sigma \in C_{\t_0} : \\ \sigma^{-1} \rho_{\t}^{-1} d \in R_{\t_0} d_{\T} \sym{\mu}}} \sgn(\sigma) =
\sum_{\substack{\sigma \in C_{\t_0} : \\ \sigma \rho_{\t}^{-1} d \in R_{\t_0} d_{\T} \sym{\mu}}} \sgn(\sigma).
\end{gather*}

We summarise this below.

\begin{Theorem} \label{T:JamesTheta}
Let $\lambda$ be a partition of $n$ and $\mu$ be a composition of $n$, and
let $\T \in \TT(\lambda,\mu)$.
The $\F\sym{n}$-module homomorphism $\hat{\theta}_{\T} \colon S_\F^{\lambda} \to M_\F^{\mu}$ constructed by James satisfies
\begin{gather*}
\hat{\theta}_{\T}(e_{\t}) = \sum_{d \sym{\mu} \in \sym{n}/\sym{\mu}} a_{\rho_\t^{-1}d,d_\T} \, d\sym{\mu},
\end{gather*} where
\begin{gather*}
a_{\rho_\t^{-1}d,d_\T} = \sum_{\substack{\sigma \in C_{\t_0} : \\ \sigma \rho_{\t}^{-1} d \in R_{\t_0} d_{\T} \sym{\mu}}} \sgn(\sigma).
\end{gather*}
\end{Theorem}

\subsection{Generalization of James's construction} \label{S:construction}
Fix a partition $\lambda$ of $n$ and a bicomposition $(\alpha|\beta)$ of $n$. In this subsection, we generalise James's construction of homomorphisms between Specht modules and Young permutation modules to obtain, for each $\dR$ in a subset $\R$ of $\sym{n}$ to be defined below (see Definition~\ref{D:RC}), a $\Z\sym{n}$-module homomorphism $\hJhom_\dR\colon S^\lambda_\Z \to M_\Z(\alpha|\beta)$ (see Theorem~\ref{T: hom}).

As before, we fix a $\lambda$-tableau $\t_0$, so that $\sym{n}$ acts on $\TT(\lambda, (\alpha|\beta))$ through $\t_0$, and denote the canonical $\lambda$-tableau of type $(\alpha|\beta)$ associated to $\t_0$ by $\T_0$. The following is the main theorem of this section:

\begin{Theorem}\label{T: hom}
	Let $\lambda$ be a partition of $n$, $(\alpha|\beta)$ be a bicomposition of $n$ and $\Gamma$ be a fixed left transversal of $\sym{\alpha|\beta}$ in $\sym{n}$. For each $\dR \in\R$, we have a $\Z\sym{n}$-module homomorphism $\hJhom_{\dR}\colon S_\Z^\lambda \to M_\Z(\alpha|\beta)$ given by
	\begin{gather*}
	\hJhom_{\dR}(e_{\t})=\sum_{d\in\Gamma} \a_{\rho_{\t}^{-1}d,\dR}\,(d \otimes \bb{1} \otimes \epsilon),\end{gather*}
	where \begin{gather*}\a_{\rho_{\t}^{-1}d, \dR} = \sum_{\substack{\sigma \in C_{\t_0} : \\ \sigma \rho_{\t}^{-1}d \in R_{\t_0} \dR \sym{\alpha|\beta} }} \sgn(\sigma)\varepsilon_\dR\big(\sigma \rho_{\t}^{-1}d\big),\end{gather*} and $\varepsilon_\dR\big(\tau \dR \xi_{\alpha}\xi_{\beta}^{+|\alpha|}\big) = \sgn(\xi_{\beta})$ for $\tau \in R_{\t_0}$ and $(\xi_{\alpha}, \xi_{\beta}) \in \sym{\alpha} \times \sym{\beta}$.
\end{Theorem}

By `reducing modulo $\Char(\F)$' the coefficients $\a_{\rho_{\t}^{-1}d, \dR}$ in $\hJhom_{\dR}(e_{\t})$, we obtain an $\F\sym{n}$-module homomorphism $\hJhom_{\dR}^\F\colon S^\lambda_\F\to M_\F(\alpha|\beta)$. Comparing this with $\hat{\theta}_{\T}$ in Theorem~\ref{T:JamesTheta}, one can see that the former is indeed a generalisation of the latter. The appearance of the map $\varepsilon_\dR$ in Theo\-rem~\ref{T: hom} also explains why our maps are indexed by elements of $\sym{n}$ instead of left cosets of $\sym{\alpha|\beta}$ (or equivalently, $\lambda$-tableaux of type $(\alpha|\beta)$), since $\varepsilon_\dR$ depends on $\dR$ and not on $\dR\sym{\alpha|\beta}$. A~little thought should convince the reader that to ensure that $\varepsilon_{\dR}$ is well-defined, $\dR$ may only run over a carefully chosen subset $\R$ of $\sym{n}$.

The remainder of this section is devoted to the proof of Theorem \ref{T: hom}.

\begin{Definition}\label{D:RC}
Define subsets of $\sym{n}$ as follows
\begin{gather*}
\R = \big\{ d \in \sym{n}\colon d^{-1} R_{\t_0}d \cap \sym{\alpha|\beta} \subseteq \sym{\alpha} \big\}, \\
\C = \big\{ d \in \sym{n}\colon d^{-1} C_{\t_0}d \cap \sym{\alpha|\beta} \subseteq \tlsym{\beta}{|\alpha|} \big\}.
\end{gather*}
\end{Definition}

These sets $\R$ and $\C$ of course depend on $\t_0$ and $(\alpha|\beta)$.

\begin{Lemma} \label{L:R}
Let $d \in \sym{n}$.
\begin{enumerate}\itemsep=0pt
\item[$(i)$] The following statements are equivalent:
\begin{enumerate}\itemsep=0pt
\item [$(a)$] $d \in \R$;
\item [$(b)$] $\stab_{R_{\t_0}} (\T_d) \subseteq d\sym{\alpha}d^{-1}$;
\item [$(c)$] whenever $\T_d(i,j) = \T_d(i,j')$ for some $(i,j),(i,j') \in [\lambda]$ with $j \ne j'$, we have $\T_d(i,j) = \bb{c}_k$ for some $k$.
\end{enumerate}
\item[$(ii)$] If $d \in \R$, then $\tau d \xi \in \R$ for all $\tau \in R_{\t_0}$ and $\xi \in \sym{\alpha|\beta}$.
\end{enumerate}
\end{Lemma}

\begin{proof} For part (i), firstly,
\begin{gather*}
\stab_{R_{\t_0}} (\T_d) = R_{\t_0} \cap \stab_{\sym{n}} (d \cdot \T_0) = R_{\t_0} \cap d \sym{\alpha|\beta} d^{-1}.
\end{gather*}
This proves the equivalence of (a) and (b).

Next, observe that a transposition $(a \;\; b)$ lies in $R_{\t_0}$ if and only if $a$ and $b$ label nodes in the same row of $\t_0$, i.e., $(\t_0)^{-1}(a) = (i,j)$ and $(\t_0)^{-1}(b) = (i,j')$ for some $i$. Furthermore, $(a \;\; b) \cdot \T_d = \T_d$ if and only if $\T_d((\t_0)^{-1}(a)) = \T_d((\t_0)^{-1}(b))$.
Since $\stab_{R_{\t_0}} (\T_d) = R_{\t_0} \cap d \sym{\alpha|\beta} d^{-1}$ is an intersection of conjugates of Young subgroups, it is generated by the transpositions it contains. Thus, $\stab_{R_{\t_0}} (\T_d)$ is generated by
\begin{gather*}
S = \big\{ (\t_0(i,j) \;\; \t_0(i,j')) \colon(i,j),(i,j') \in [\lambda],\, j \ne j',\, \T_d(i,j) = \T_d(i,j') \big\}.
\end{gather*}

Now,
\begin{gather*}
 (\t_0(i,j) \;\; \t_0(i,j')) \in d\sym{\alpha} d^{-1} \\
\quad {} \Leftrightarrow \big( d^{-1}(\t_0(i,j)) \;\; d^{-1}(\t_0(i,j'))\big) \in \sym{\alpha} \\
\quad {}\Leftrightarrow \exists\, (1 \leq k \leq \ell(\alpha)), \
\sum_{i=1}^{k-1} \alpha_i < d^{-1}(\t_0(i,j)), \ d^{-1}(\t_0(i,j')) \leq \sum_{i=1}^{k} \alpha_i\\
\quad {}\Leftrightarrow \exists\, (1 \leq k \leq \ell(\alpha)), \ \T_0 \big((\t_0)^{-1} \big(d^{-1}(\t_0(i,j))\big)\big) = \T_0 \big((\t_0)^{-1} \big( d^{-1}(\t_0(i,j'))\big)\big) = \bb{c}_k \\
\quad {} \Leftrightarrow \exists\, (1 \leq k \leq \ell(\alpha)), \ \T_d (i,j) = \T_d (i,j') = \bb{c}_k .
\end{gather*}
Hence (b) and (c) are equivalent.

For part (ii), observe that
\begin{gather*}
(\tau d \xi)^{-1} R_{\t_0} (\tau d \xi) \cap \sym{\alpha|\beta} = \xi^{-1}\big(d^{-1}R_{\t_0}d\cap \sym{\alpha|\beta}\big)\xi \subseteq
\xi^{-1}\sym{\alpha}\xi = \sym{\alpha}.\tag*{\qed}
\end{gather*}\renewcommand{\qed}{}
\end{proof}

We have analogous statements and proofs for $\C$ too.

\begin{Lemma} \label{L:C}
Let $d \in \sym{n}$.
\begin{enumerate}\itemsep=0pt
\item[$(i)$] The following statements are equivalent:
\begin{enumerate}\itemsep=0pt
\item [$(a)$] $d \in \C$;
\item [$(b)$] $\stab_{C_{\t_0}} (\T_d) \subseteq d\tlsym{\beta}{|\alpha|}d^{-1}$;
\item [$(c)$] whenever $\T_d(i,j) = \T_d(i',j)$ for some $(i,j),(i',j) \in [\lambda]$ with $i \ne i'$, we have $\T_d(i,j) = \bb{d}_k$ for some $k$.
\end{enumerate}
\item[$(ii)$] If $d \in \C$, then $\tau d \xi \in \C$ for all $\tau \in C_{\t_0}$ and $\xi \in \sym{\alpha|\beta}$.
\end{enumerate}
\end{Lemma}

Lemmas \ref{L:R} and \ref{L:C} give the following immediate corollary.

\begin{Corollary}\label{C:sstdRC}
If $d \in \sym{n}$ such that $\T_d \in \TTS(\lambda, (\alpha|\beta))$, then $d \in \R \cap \C$.
\end{Corollary}

In order to generalise the coefficient $a_{d,d_\T}$ in Theorem \ref{T:JamesTheta} to $\a_{d,\dR}$ in Definition \ref{D:Omega&a}, we need the following lemma which also explains the choice of the set $\R$.

\begin{Lemma} \label{L:projection}
Let $\dR \in \R$. There is a well-defined projection map $\pi_\dR\colon R_{\t_0}\dR\sym{\alpha|\beta} \to \sym{\beta}$ defined by $\tau \dR \xi_{\alpha} \tlxi{\beta}{|\alpha|} \mapsto \xi_{\beta}$ for all $\tau \in R_{\t_0}$ and $(\xi_\alpha,\xi_\beta) \in \sym{\alpha} \times \sym{\beta}$.
\end{Lemma}

\begin{proof}If $\tau \dR \xi_{\alpha}\tlxi{\beta}{|\alpha|} = \tau' \dR \xi'_{\alpha} \xi'^{+|\alpha|}_{\beta}$, then
\begin{gather*}\dR^{-1} R_{\t_0} \dR \ni \dR^{-1} \big(\tau^{-1} \tau'\big)\dR = \xi_{\alpha}\tlxi{\beta}{|\alpha|} \big(\xi'_{\alpha} \xi'^{+|\alpha|}_{\beta}\big)^{-1} \in \sym{\alpha|\beta},\end{gather*}
so that $\xi_{\alpha}\tlxi{\beta}{|\alpha|} \big(\xi'_{\alpha} \xi'^{+|\alpha|}_{\beta}\big)^{-1} \in \sym{\alpha}$ since $\dR \in \R$, forcing $\xi_{\beta} = \xi'_{\beta}$. The lemma thus follows.
\end{proof}

In view of Lemma \ref{L:projection}, we can make the following definition.

\begin{Definition}\label{D:Omega&a}
Let $d, \dR \in \sym{n}$ with $\dR \in \R$.
\begin{enumerate}\itemsep=0pt
\item [(i)] For $\omega \in R_{\t_0} \dR \sym{\alpha|\beta}$, let $\e_{\dR}(\omega) :=\sgn(\pi_{\dR}(\omega))\in \{\pm1\}$.
\item [(ii)] Let
\begin{gather*}
\Omega_{d,\dR} := \big\{ \sigma \in C_{\t_0}\colon \sigma d \in R_{\t_0} \dR \sym{\alpha|\beta} \big\}, \qquad
\a_{d,\dR} := \sum_{\sigma \in \Omega_{d,\dR} } \sgn(\sigma) \e_{\dR}(\sigma d) \in \Z.
\end{gather*} By convention, if $\Omega_{d,\dR}=\varnothing$ then $\a_{d,\dR}=0$.
\end{enumerate}
\end{Definition}

\begin{Remark} \label{Omega}
We give another description of $\Omega_{d,\dR}$ here. Restrict the left regular action of the symmetric group $\sym{n}$ on $\sym{n}/\sym{\alpha|\beta}$ to the subgroups $R_{\t_0}$ and $C_{\t_0}$, which partition $\sym{n}/\sym{\alpha|\beta}$ into $R_{\t_0}$-orbits and into $C_{\t_0}$-orbits respectively. Then $\sigma \in \Omega_{d,\dR}$ if and only if $\sigma \in C_{\t_0}$ and $\sigma \cdot (d\sym{\alpha|\beta}) \in R_{\t_0} \cdot (\dR\sym{\alpha|\beta})$. As such, $\Omega_{d,\dR} \ne \varnothing$ if and only if $C_{\t_0}d\sym{\alpha|\beta} \cap R_{\t_0} \dR\sym{\alpha|\beta} \ne \varnothing$. Furthermore, if $C_{\t_0}d\sym{\alpha|\beta} \cap R_{\t_0} \dR\sym{\alpha|\beta}$ contains precisely the distinct left cosets $d^{(1)}\sym{\alpha|\beta}, d^{(2)}\sym{\alpha|\beta}, \dotsc, d^{(r)}\sym{\alpha|\beta}$, then $\Omega_{d,\dR} = \bigcup\limits_{i=1}^r \Omega^{(i)}$ (disjoint union), where for each $i$, $\Omega^{(i)} = \{ \sigma \in C_{\t_0}\colon \sigma \cdot d\sym{\alpha|\beta} = d^{(i)}\sym{\alpha|\beta} \}$ and is therefore a left coset of $\stab_{C_{\t_0}} (d\sym{\alpha|\beta}) = C_{\t_0} \cap d\sym{\alpha|\beta} d^{-1}$. In particular, $\Omega_{d,\dR}$ is a union of some left cosets of $C_{\t_0} \cap d\sym{\alpha|\beta} d^{-1}$.
\end{Remark}

We collect together some important properties that $\a_{d,\dR}$ satisfies:

\begin{Lemma} \label{L:a}
Let $d,\dR \in \sym{n}$ with $\dR \in \R$.
\begin{enumerate}\itemsep=0pt
\item [$(i)$] If $\dR' = \tau \dR \xi_{\alpha}\tlxi{\beta}{|\alpha|}$ and $d' = \sigma d \eta_{\alpha}\eta_{\beta}^{+|\alpha|}$ for some $\tau\in R_{\t_0}$, $\sigma\in C_{\t_0}$ and $(\xi_\alpha,\xi_\beta),(\eta_\alpha,\eta_\beta) \in \sym{\alpha} \times \sym{\beta}$, then $\dR' \in \R$ and
 \begin{gather*}
 \a_{d',\dR'} = \sgn(\sigma)\sgn(\xi_\beta)\sgn(\eta_{\beta}) \a_{d,\dR}.
 \end{gather*}

\item [$(ii)$] If $C_{\t_0}d\sym{\alpha|\beta} \cap R_{\t_0} \dR \sym{\alpha|\beta} = \varnothing$ or $d \notin \C$, then $\a_{d,\dR} = 0$.
\item [$(iii)$] Suppose that $d \in \C$ and that $C_{\t_0}d\sym{\alpha|\beta} \cap R_{\t_0} \dR \sym{\alpha|\beta}$ is a disjoint union of $r$ left cosets of~$\sym{\alpha|\beta}$, with representatives $d^{(1)},\dotsc, d^{(r)}$. For each $i$, let $\sigma_i \in C_{\t_0}$ such that $\sigma_i d \in d^{(i)} \sym{\alpha|\beta}$, and let $\e^{(i)} = \sgn(\sigma_i) \e_{\dR}(\sigma_i d)$. Then
 \begin{gather*}
 \a_{d,\dR} = \big|C_{\t_0} \cap d \sym{\alpha|\beta} d^{-1}\big| \sum_{i=1}^r \e^{(i)}.
 \end{gather*}
\end{enumerate}
\end{Lemma}

\begin{proof} By Lemma \ref{L:R}(ii), $\dR' \in \R$. It is also easy to see that \begin{gather*}\Omega_{d',\dR'} = \Omega_{\sigma d,\dR}= \Omega_{d,\dR}\sigma^{-1}.\end{gather*} Take $\omega \in \Omega_{d',\dR'}$, say $\omega d'=\tau' \dR'\gamma_\alpha\gamma_\beta^{+|\alpha|}$ with $\tau'\in R_{\t_0}$ and $(\gamma_\alpha,\gamma_\beta) \in \sym{\alpha} \times \sym{\beta}$. Then \begin{gather*}\omega \sigma d=\tau' \tau \dR \xi_{\alpha}\gamma_\alpha\eta_{\alpha}^{-1}\big(\xi_\beta\gamma_\beta \eta_\beta^{-1}\big)^{+|\alpha|}\end{gather*} and hence we have $\e_{\dR'}(\omega d')=\sgn(\gamma_\beta)=\e_{\dR}(\omega \sigma d)\sgn(\eta_\beta)\sgn(\xi_\beta)$. Therefore{\samepage
\begin{align*}
\a_{d',\dR'} &= \sum_{\omega \in \Omega_{d',\dR'}} \sgn(\omega) \e_{\dR'}(\omega d')
= \sum_{\omega\sigma \in \Omega_{d,\dR}} \sgn(\sigma)\sgn(\omega\sigma) \e_{\dR}(\omega\sigma d)\sgn(\xi_{\beta})\sgn(\eta_{\beta}) \\
&= \sgn(\sigma)\sgn(\xi_{\beta})\sgn(\eta_{\beta}) \a_{d,\dR}.
\end{align*} This completes the proof of part (i).}

For part (ii), it is clear that $\Omega_{d,\dR} = \varnothing$ if $C_{\t_0}d\sym{\alpha|\beta} \cap R_{\t_0} \dR \sym{\alpha|\beta} = \varnothing$, so that $\a_{d,\dR} = 0$ in this instance. Next we assume that $d \notin \C$. By Lemma \ref{L:C}(i), there exist $(i,j), (i',j) \in [\lambda]$ with $i\ne i'$ such that $\T_d(i,j) = \bb{c}_k = \T_d(i',j)$ for some $k$. Let $\rho = (\t_0(i,j) \;\; \t_0(i',j))$. Then $\rho \in C_{\t_0} \cap d\sym{\alpha} d^{-1}$. If $\sigma \in \Omega_{d,\dR}$, say $\sigma d = \tau \dR \gamma_\alpha \gamma_\beta^{+|\alpha|}$ with $\tau \in R_{\t_0}$ and $(\gamma_\alpha,\gamma_\beta) \in \sym{\alpha} \times \sym{\beta}$, then
\begin{gather*}
R_{\t_0} \dR \sym{\alpha|\beta} \ni \tau \dR \gamma_{\alpha} \big(d^{-1}\rho d\big) \gamma_{\beta}^{+|\alpha|} = \sigma d \big(d^{-1}\rho d\big) = (\sigma \rho) d, \end{gather*}
so that $\sigma \rho \in \Omega_{d,\dR}$, and
\begin{gather*}
\sgn(\sigma \rho) \e_{\dR}(\sigma \rho d) = \sgn(\rho)\sgn(\sigma)\sgn(\gamma_\beta)= - \sgn(\sigma) \e_{\dR}(\sigma d).
\end{gather*}
As such, the contributions to the sum in $\a_{d,\dR}$ by $\sigma$ and $\sigma \rho$ cancel each other out. Consequently, $\a_{d,\dR} = 0$ and the proof of part (ii) is now complete.

For part (iii), for each $i$, let $\Omega^{(i)} = \{ \sigma \in \Omega_{d,\dR} \colon \sigma d \in d^{(i)}\sym{\alpha|\beta} \}$, so that $\Omega_{d,\dR}$ is a disjoint union of the $\Omega^{(i)}$'s (see Remark~\ref{Omega}). Fix $i$. There exist $\tau_i \in R_{\t_0}$ and $\big(\gamma^{(i)}_{\alpha},\gamma_{\beta}^{(i)}\big)\in\sym{\alpha}\times \sym{\beta}$ such that $\sigma_i d = \tau_i\dR \gamma_{\alpha}^{(i)} \tl{\big(\gamma_{\beta}^{(i)}\big)}{|\alpha|}$. For any $\omega \in \Omega^{(i)}$, we have
\begin{gather*} d^{-1} C_{\t_0} d \ni d^{-1} \big(\sigma_i^{-1}\omega\big) d = (\sigma_i d)^{-1} (\omega d) \in \sym{\alpha|\beta},\end{gather*}
so that $d^{-1} \big(\sigma_i^{-1}\omega\big) d \in \tlsym{\beta}{|\alpha|}$ since $d \in \C$. Thus
\begin{gather*}
\omega d = \sigma_i d \big(d^{-1} \sigma_i ^{-1} \omega d\big) = \tau_i \dR \gamma^{(i)}_{\alpha} \big(\gamma^{(i)}_{\beta}\big)^{+|\alpha|}\big( d^{-1} \sigma_i^{-1} \omega d\big),\end{gather*}
so that $\pi_{\dR}(\omega d) = \pi_{\dR}(\sigma_i d) d^{-1} \sigma_i^{-1} \omega d$ and hence
\begin{gather*}
\sgn(\omega) \e_{\dR}(\omega d) = \sgn(\omega) \sgn\big(\pi_{\dR}(\sigma_i d) d^{-1} \sigma_i^{-1} \omega d\big) = \sgn(\sigma_i)\e_{\dR}(\sigma_i d) = \e^{(i)}.\end{gather*}
Consequently,
\begin{align*}\allowdisplaybreaks
\a_{d,\dR} &= \sum_{\omega \in \Omega_{d,\dR}} \sgn(\omega) \e_{\dR}(\omega d)
= \sum_{i=1}^r \sum_{\omega \in \Omega^{(i)}} \sgn(\omega) \e_{\dR}(\omega d)
= \sum_{i=1}^r \sum_{\omega \in \Omega^{(i)}} \e^{(i)} \\
&= \sum_{i=1}^r |\Omega^{(i)}| \e^{(i)}
= \big|C_{\t_0} \cap d \sym{\alpha|\beta}d^{-1}\big| \sum_{i=1}^r \e^{(i)},
\end{align*} where the final equality is given by the following bijection: fix $\sigma^{(i)}\in \Omega^{(i)}$, we have a bijection between the sets $C_{\t_0} \cap d \sym{\alpha|\beta}d^{-1}$ and $\Omega^{(i)}$ given by $\sigma'\mapsto \sigma^{(i)}\sigma'$.
\end{proof}

\begin{Example} \label{E:a}
We continue with Example \ref{E:ss}(ii), where $\lambda = (2,1^{p+2})$, $(\alpha|\beta) = (\varnothing|(p,2^2))$ and~$\T_1$, $\T_2$, $\T_3$ are the $\lambda$-tableaux of type $(\alpha|\beta)$ as given in the example. Let $\t_0$ be the $\lambda$-tableau defined by $\t_0(j,1) = j$ for $1 \leq j \leq p+3$ and $\t_0(1,2) = p+4$. Then $\T_0 = \T_3$. Furthermore, under the action of $\sym{p+4}$ on $\TT(\lambda, (\alpha|\beta))$, there are exactly three $C_{\t_0}$-orbits, with orbit representatives~$\T_1$,~$\T_2$ and~$\T_3$.

Let $d_i$ ($i \in \{1,2,3\}$) be the following permutations in $\sym{p+4}$: $d_1 = (p\ \ p+4\ \ p+2)$, $d_2 = (p+2\ \ p+4)$ and $d_3 = 1_{\sym{p+4}}$. Then $d_i \cdot \T_0 = \T_i$ for all $i$. Using Lemma \ref{L:a}(iii), we get
\begin{alignat*}{4}
& \a_{d_1,d_2} = (p-1)! \cdot 4, \qquad && \a_{d_2,d_2} = p! \cdot 2, \qquad&& \a_{d_3,d_2}= 0,& \\
& \a_{d_1,d_3} = -(p-1)! \cdot 4, \qquad && \a_{d_2,d_3} = 0, \qquad && \a_{d_3,d_3} = p!\cdot 2.&
\end{alignat*}
From Lemma \ref{L:a}(i), we conclude further that $\a_{d,d_2}, \a_{d,d_3} \in \{ \pm ((p-1)! \cdot 4), \pm (p! \cdot 2), 0 \}$ for all $d \in \sym{p+4}$.
\end{Example}

We are now ready to define, for each $\dR\in \R$, a map $\Jhom_{\dR}$ which will eventually lead to $\hJhom_\dR$.

\begin{Definition}\label{D:theta}
Let $\lambda$ be a partition of $n$. Recall that there is a natural left $\sym{n}$-action on the set $\TT(\lambda)$ of $\lambda$-tableaux, and let $\Z\TT(\lambda)$ denote the associated permutation $\Z\sym{n}$-module. Let $(\alpha|\beta)$ be a bicomposition of $n$ and let $\Gamma$ be a fixed left transversal of $\sym{\alpha|\beta}$ in $\sym{n}$. For each $\dR \in\R$, we define a $\Z$-linear map $\Jhom_{\dR}\colon \Z\TT(\lambda) \to M_\Z(\alpha|\beta)$ as follows
\begin{gather*}\Jhom_{\dR}(\t) = \sum_{d \in \Gamma} \a_{\rho_{\t}^{-1} d,\dR} (d \otimes \bb{1} \otimes \epsilon),\end{gather*}
where $\rho_{\t} = \t \circ (\t_0)^{-1} \in \sym{n}$ (so that $\rho_{\t} \cdot \t_0 = \rho_{\t} \circ \t_0 = \t$).
\end{Definition}

The following properties of the map $\Jhom_{\dR}$ follow easily from its definition and Lem\-ma~\ref{L:a}(i).

\begin{Lemma} \label{L:Jhom}
Let $\dR \in \R$.
\begin{enumerate}\itemsep=0pt
 \item [$(i)$] The map $\Jhom_{\dR}$ is independent of the choice of the left transversal $\Gamma$, i.e., if $\Gamma'$ is any left transversal of $\sym{\alpha|\beta}$ in $\sym{n}$, then
 \begin{gather*}
 \Jhom_{\dR}(\t) = \sum_{d' \in \Gamma'} \a_{\rho_{\t}^{-1} d',\dR} (d' \otimes \bb{1} \otimes \epsilon)\end{gather*} for all $\t \in \TT(\lambda)$.
 \item [$(ii)$] If $\dR' \in R_{\t_0} \dR \sym{\alpha|\beta}$, then $\Jhom_{\dR'} = \e_{\dR}(\dR') \Jhom_{\dR}$.
\end{enumerate}
\end{Lemma}

\begin{proof} Let $d' = d \eta_{\alpha} \eta_{\beta}^{+|\alpha|}$ with $(\eta_{\alpha},\eta_{\beta}) \in \sym{\alpha} \times \sym{\beta}$. Then $\a_{\rho_{\t}^{-1}d',\dR} = \sgn(\eta_\beta) \a_{\rho_{\t}^{-1}d,\dR}$ by Lem\-ma~\ref{L:a}(i) while
\begin{gather*}d' \otimes \bb{1} \otimes \epsilon = (d \eta_{\alpha} \eta_{\beta}) \otimes \bb{1} \otimes \epsilon = d \otimes (\eta_{\alpha} \cdot \bb{1}) \otimes (\eta_{\beta} \cdot \epsilon) = \sgn(\eta_{\beta}) (d \otimes \bb{1} \otimes \epsilon).\end{gather*}
Thus $\a_{\rho_{\t}^{-1}d',\dR} (d' \otimes \bb{1} \otimes \epsilon) = \a_{\rho_{\t}^{-1}d,\dR} (d \otimes \bb{1} \otimes \epsilon)$ and part (i) follows.

By Lemma \ref{L:a}(i), $\a_{\rho_{\t}^{-1}d,\dR'} = \e_{\dR}(\dR') \a_{\rho_{\t}^{-1}d,\dR}$. Part (ii) thus follows.
\end{proof}

Next, we aim to show that each of the maps $\Jhom_{\dR}\colon \Z\TT(\lambda)\to M_\Z(\alpha|\beta)$ induces a $\Z\sym{n}$-module homomorphism $\hJhom_{\dR}\colon S^\lambda_\Z\to M_\Z(\alpha|\beta)$. For this, we will show that $\Jhom_{\dR}$ is a $\Z\sym{n}$-module homomorphism and that the kernel of the natural map $\psi\colon \Z\TT(\lambda)\to S^\lambda_\Z$ given by $\psi(\t)=e_\t$ is contained in $\ker(\Jhom_{\dR})$ and hence $\Jhom_{\dR}$ induces $\hJhom_{\dR}$ as desired.

For each $j$, let $C_j(\lambda) = \{ (i,j) \in [\lambda] \colon 1 \leq i \leq \ell(\lambda) \}$ be the $j$th column of the Young diagram of $\lambda$. A {\em Garnir transversal} $\Delta$ is a left transversal of $\sym{X}\sym{Y}$ in $\sym{X \cup Y}$, where $\varnothing \ne X \subseteq C_j(\lambda)$ and $\varnothing \ne Y \subseteq C_{j'}(\lambda)$ with $j < j'$, such that $|X|+|Y| > |C_j(\lambda)|$. Given a Garnir transversal $\Delta$ and a $\lambda$-tableau $\t$, let $\Delta_{\t} = \{ \t \circ \gamma \circ \t^{-1} \colon \gamma \in \Delta \}$, so that $\Delta_{\t}$ is a left transversal of $\sym{\t(X)}\sym{\t(Y)}$ in~$\sym{\t(X \cup Y)}$, and write
\begin{gather*}G_{\Delta}^{\t} = \sum_{\gamma \in \Delta_\t} \sgn(\gamma) \gamma.\end{gather*}

\begin{Proposition}\label{P: phi_d hom}
Let $\dR \in \R$. Then
\begin{enumerate}\itemsep=0pt
\item [$(i)$] $\Jhom_{\dR}$ is a $\Z \sym{n}$-module homomorphism, and
\item [$(ii)$] for any $\t \in \TT(\lambda)$, $\pi\in C_\t$ and Garnir transversal $\Delta$, we have that $\ker(\Jhom_{\dR})$ contains both $\pi \cdot \t - \sgn(\pi) \t$ and $G_{\Delta}^{\t} \cdot \t$.
\end{enumerate}
\end{Proposition}

\begin{proof} Let $\t \in \TT(\lambda)$ and $x \in \sym{n}$. We have
\begin{gather*}
\Jhom_{\dR}(x \cdot \t) = \Jhom_{\dR}((x \rho_{\t}) \cdot \t_0) = \sum_{d \in \Gamma} \a_{(x \rho_{\t})^{-1} d, \dR} (d \otimes \bb{1} \otimes \epsilon), \\
x \cdot \Jhom_{\dR} (\t) = x \cdot \left (\sum_{d \in \Gamma} \a_{\rho_{\t}^{-1} d, \dR} (d \otimes \bb{1} \otimes \epsilon)\right ) = \sum_{d \in \Gamma} \a_{\rho_{\t}^{-1} d, \dR} \sgn(\xi_{d,\beta}) (f(d) \otimes \bb{1} \otimes \epsilon),
\end{gather*}
where $xd = f(d) \xi_{d,\alpha} \xi_{d,\beta}^{+|\alpha|}$ with $f(d) \in \Gamma$, $(\xi_{d,\alpha},\xi_{d,\beta}) \in \sym{\alpha} \times \sym{\beta}$. As such, we need to show that for each $d \in \Gamma$,
\begin{gather*}
\a_{(x \rho_{\t})^{-1} f(d), \dR} = \a_{\rho_{\t}^{-1} d, \dR} \sgn(\xi_{d,\beta}).
\end{gather*}
Observe that $\sigma (x\rho_{\t})^{-1}f(d)=\tau \dR \eta_\alpha \eta_\beta^{+|\alpha|}$ for some $\sigma\in C_{\t_0}$, $\tau\in R_{\t_0}$ and $(\eta_\alpha,\eta_\beta)\in\sym{\alpha}\times \sym{\beta}$ if and only if $\sigma\rho_{\t}^{-1}d=\tau \dR (\eta_\alpha \xi_{d,\alpha}) (\eta_\beta\xi_{d,\beta})^{+|\alpha|}$ for some $\sigma\in C_{\t_0}$, $\tau\in R_{\t_0}$ and $(\eta_\alpha,\eta_\beta)\in\sym{\alpha}\times \sym{\beta}$, and in which case, $\e_{\dR}\big( \sigma (\rho_{\t})^{-1}d\big) = \sgn(\xi_{d,\beta}) \e_{\dR}\big(\sigma(x\rho_{\t})^{-1}f(d)\big)$. So $\Omega_{(x\rho_{\t})^{-1}f(d),\dR}=\Omega_{\rho_{\t}^{-1}d,\dR}$, and
\begin{align*}
\a_{(x\rho_{\t})^{-1} f(d),\dR} &= \sum_{\sigma\in \Omega_{(x\rho_{\t})^{-1}f(d),\dR}} \sgn(\sigma)\e_{\dR}\big(\sigma (x\rho_{\t})^{-1} f(d)\big)\\
&=\sum_{\sigma\in \Omega_{\rho_{\t}^{-1}d,\dR}} \sgn(\sigma)\sgn(\xi_{d,\beta})\e_{\dR}\big(\sigma \rho_{\t}^{-1}d\big)
=\a_{\rho_{\t}^{-1}d,\dR} \sgn(\xi_{d,\beta}).
\end{align*} This proves part (i).

Since $\rho_\t\cdot \t_0=\t$, we have $C_\t=\rho_\t C_{\t_0}\rho_\t^{-1}$ and hence $\rho_\t^{-1}\pi^{-1}\rho_\t\in C_{\t_0}$. By Lemma \ref{L:a}(i), we have \begin{gather*}\a_{(\pi\rho_{\t})^{-1} d, \dR}=\a_{\rho_\t^{-1}\pi^{-1}\rho_\t\rho^{-1}_\t d, \dR} = \sgn(\pi) \a_{\rho_{\t}^{-1}d, \dR}.\end{gather*} Therefore
\begin{gather*}
\Jhom_{\dR}(\pi \cdot \t) = \sum_{d \in \Gamma} \a_{(\pi \rho_{\t})^{-1}d, \dR} (d \otimes \bb{1} \otimes \epsilon)
=\sum_{d \in \Gamma} \sgn(\pi) \a_{\rho_{\t}^{-1}d, \dR} (d \otimes \bb{1} \otimes \epsilon)
=\sgn(\pi)\Jhom_{\dR}(\t).
\end{gather*}

Next we turn to $G_{\Delta}^{\t} \cdot \t$. We have
\begin{align*}
\Jhom_{\dR}\big(G_{\Delta}^{\t} \cdot \t\big) & = \Jhom_{\dR}\left (\sum_{\gamma \in \Delta_{\t}} \sgn(\gamma) (\gamma \cdot \t)\right )
= \sum_{\gamma \in \Delta_{\t}} \sum_{d \in \Gamma} \sgn(\gamma) \a_{(\gamma \rho_{\t})^{-1} d, \dR} (d \otimes \bb{1} \otimes \epsilon) \\
& = \sum_{d \in \Gamma} \left( \sum_{\gamma \in \Delta_{\t}} \sum_{\sigma \in \Omega_{(\gamma \rho_{\t})^{-1} d, \dR}} \sgn(\gamma) \sgn(\sigma) \e_{\dR}\big( \sigma (\gamma \rho_{\t})^{-1} d\big) \right) (d \otimes \bb{1} \otimes \epsilon).
\end{align*}
Fix $d \in \Gamma$, and let $b(\gamma,\sigma) = \sgn(\gamma) \sgn(\sigma) \e_{\dR}\big( \sigma (\gamma \rho_{\t})^{-1} d\big)$. We need to show that $\!\sum\limits_{(\gamma,\sigma) \in \Upsilon} \! b(\gamma,\sigma) =0$, where $\Upsilon = \{ (\gamma,\sigma) \colon \gamma \in \Delta_{\t},\, \sigma \in \Omega_{(\gamma \rho_{\t})^{-1} d, \dR} \}$.

Let the Garnir transversal $\Delta$ be a left transversal of $\sym{X}\sym{Y}$ in $\sym{X \cup Y}$, where $X$ and $Y$ are subsets of the $j$th and $j'$th column of $[\lambda]$, with $|X| + |Y| > |C_j(\lambda)|$. Since $|X| + |Y| > |C_j(\lambda)|$, there exists $1 \leq i \leq \ell(\lambda)$ such that $(i,j) \in X$ and $(i,j') \in Y$, and we may choose $i$ to be the least such. Let $\eta$ be the transposition $(\t_0(i,j) \;\; \t_0(i,j'))$. Then $\eta \in R_{\t_0}$.

Let $(\gamma, \sigma) \in \Upsilon$, say $\sigma (\gamma \rho_{\t})^{-1} d = \tau \dR \xi_\alpha \tlxi{\beta}{|\alpha|}$ where $\tau \in R_{\t_0}$ and $(\xi_\alpha,\xi_\beta) \in \sym{\alpha} \times \sym{\beta}$. Then
\begin{align*}
R_{\t_0} \dR \sym{\alpha|\beta} \ni \eta \tau \dR \xi_\alpha \tlxi{\beta}{|\alpha|} &=
\eta \sigma (\gamma \rho_{\t})^{-1} d = \eta \sigma \rho_{\t}^{-1} \gamma^{-1} d\\
&= \big(\sigma \rho_{\t}^{-1}\big) \big(\big(\sigma \rho_{\t}^{-1}\big)^{-1} \eta \big(\sigma \rho^{-1}_{\t} \big)\big) \gamma^{-1} d
= \sigma \rho_{\t}^{-1} \zeta^{-1} \rho_{\t} (\gamma'\rho_{\t})^{-1} d,
\end{align*}
where $\gamma' \in \Delta_{\t}$ and $\zeta \in \sym{\t(X)}\sym{\t(Y)}$ satisfies $\gamma \big(\big(\sigma \rho_{\t}^{-1}\big)^{-1} \eta \big(\sigma \rho^{-1}_{\t} \big)\big) = \gamma' \zeta$. Then $\rho_{\t}^{-1} \zeta^{-1} \rho_{\t} \in \rho_{\t}^{-1} \sym{\t(X)} \sym{\t(Y)} \rho_{\t} = \sym{\t_0(X)}\sym{\t_0(Y)} \subseteq C_{\t_0}$. This shows that $(\gamma',\sigma \rho_{\t}^{-1} \zeta^{-1} \rho_{\t}) \in \Upsilon$. In other words, the function $h \colon \Upsilon \to \Upsilon$ defined by $(\gamma,\sigma) \mapsto \big(\gamma',\sigma \rho_{\t}^{-1} \zeta^{-1} \rho_{\t}\big)$, where $\gamma \big(\big(\sigma \rho_{\t}^{-1}\big)^{-1} \eta \big(\sigma \rho^{-1}_{\t} \big)\big) = \gamma' \zeta$, is well-defined. Furthermore,
\begin{align*}
b(h(\gamma,\sigma)) &= \sgn(\gamma')\sgn\big(\sigma \rho_{\t}^{-1} \zeta^{-1} \rho_{\t} \big) \e_{\dR}\big(\sigma \rho_{\t}^{-1} \zeta^{-1} \rho_{\t} (\gamma'\rho_{\t})^{-1} d\big) \\
&= -\sgn(\gamma) \sgn(\sigma) \sgn(\xi_{\beta}) = -b(\gamma,\sigma).
\end{align*}
We claim that $h$ is a fixed-point-free involution, in which case, since the contributions from $(\gamma,\sigma)$ and $h(\gamma,\sigma)$ towards the sum $\sum\limits_{(\gamma,\sigma) \in \Upsilon} b(\gamma,\sigma)$ cancel each other out, we have $\sum\limits_{(\gamma,\sigma) \in \Upsilon} b(\gamma,\sigma) =0$ as desired.

To prove the claim, first suppose that
\begin{gather*}
(\gamma,\sigma) =h(\gamma,\sigma)=\big(\gamma',\sigma \rho_{\t}^{-1} \zeta^{-1} \rho_{\t}\big).\end{gather*}
Then $\gamma=\gamma'$ and $\zeta=1$ and hence $\eta=1$, a contradiction. So $h$ is fixed-point-free. Next, $h^2(\gamma,\sigma) = h\big(\gamma',\sigma \rho_{\t}^{-1} \zeta^{-1} \rho_{\t}\big) = \big(\gamma'', \sigma \rho_{\t}^{-1} \zeta^{-1} \rho_{\t} \rho_{\t}^{-1} \zeta'^{-1} \rho_{\t}\big)$, where
\begin{gather*}
\gamma''\zeta' = \gamma' \big(\sigma \rho_{\t}^{-1} \zeta^{-1}\big)^{-1} \eta \big(\sigma \rho_{\t}^{-1} \zeta^{-1}\big) =
\gamma' \zeta \big(\sigma \rho_{\t}^{-1}\big)^{-1} \eta \big(\sigma \rho_{\t}^{-1}\big) \zeta^{-1} = \gamma \zeta^{-1}.
\end{gather*}
Thus $\gamma'' = \gamma$ and $\zeta' = \zeta^{-1}$ and hence $h^2(\gamma,\sigma) = (\gamma, \sigma)$, and the proof is complete.
\end{proof}

The next result is well known when the underlying ring is a field (see, for example, \cite[Section~7.4, Corollary, p.~101]{WF}); we are however unable to find its generalisation to~$\mathbb{Z}$ in the existing literature.

\begin{Lemma}\label{L: annihilator for specht}
The map $\psi\colon \Z\TT(\lambda) \to S_\Z^\lambda$ defined by $\t \mapsto e_{\t}$ is a $\Z\sym{n}$-module epimorphism, and $\ker(\psi)$ is generated, as a $\Z$-submodule of $\Z\TT(\lambda)$, by $G \cup H$, where
\begin{gather*}
G= \big\{ G_{\Delta}^{\t} \cdot \t \colon \t \in \TT(\lambda),\ \Delta \text{ a Garnir transversal} \big\}, \\
H= \{ \pi \cdot \t - \sgn(\pi)\t \colon \t \in \TT(\lambda),\ \pi \in C_{\t} \}.
\end{gather*}
\end{Lemma}

\begin{proof}That $\psi$ is a $\Z\sym{n}$-module epimorphism is clear, so we only need to justify the assertion about its kernel. Let $K = \Z(G\cup H)$. It is straightforward to verify that both $G$ and $H$ are invariant under the action of $\sym{n}$, and that $\psi(G \cup H) = \{0\}$ \cite[7.2.1 and Theorem~7.2.3]{GJAK}, so that~$K$ is a $\Z\sym{n}$-submodule of $\Z\TT(\lambda)$, and $K \subseteq \ker(\psi)$. It remains to show that $\ker(\psi) \subseteq K$.

Since $\psi(\Z\TT(\lambda)) = S^{\lambda}_\Z$, and $S_\Z^{\lambda}$ is $\Z$-free with basis $\{ e_{\t} \colon \t \in \TT_{\mathrm{std}}(\lambda) \}$ \cite[Theorem~1.1]{Peel}, we see that $\Z\TT(\lambda) = (\bigoplus_{\t \in \TT_{\mathrm{std}}(\lambda)} \Z\t)\oplus \ker(\psi)$.

Consider $\Z\TT(\lambda)/K$. For each $\t \in \TT(\lambda)$, let $v_{\t} = \t + K \in \Z\TT(\lambda)/K$. Then $\pi\cdot v_\t = \sgn(\pi)v_\t$ for any $\pi \in C_{\t}$, and $G_{\Delta}^{\t} \cdot v_\t = 0$ for any Garnir transversal $\Delta$. Using the same argument as in the proof of \cite[Theorem~7.2.7]{GJAK}, we can show that $\{ v_{\t} \colon \t \in \TT_{\mathrm{std}}(\lambda) \}$ generates $\Z\TT(\lambda)/K$ as a $\Z$-module.

Let $x \in \ker(\psi)$. Then there exists $b_{\t} \in \Z$ for each $\t \in \TT_{\mathrm{std}}(\lambda)$ such that \begin{gather*}x+K = \sum_{\t \in \TT_{\mathrm{std}}(\lambda)} b_{\t}v_{\t} = \sum_{\t \in \TT_{\mathrm{std}}(\lambda)} b_{\t}(\t + K).\end{gather*}
Thus there exists $k \in K$ such that \begin{gather*}x + k = \sum_{\t \in \TT_{\mathrm{std}}(\lambda)} b_{\t}\t \in \left (\bigoplus_{\t \in \TT_{\mathrm{std}}(\lambda)} \Z\t\right ) \cap \ker(\psi) = \{0\}.\end{gather*} Hence $x + k = 0$ and so $x \in K$, and our proof is complete.
\end{proof}

We are now ready to prove Theorem \ref{T: hom}.

\begin{proof}[Proof of Theorem \ref{T: hom}]
By Proposition \ref{P: phi_d hom} and Lemma \ref{L: annihilator for specht}, we have $\ker(\psi) \subseteq \ker(\Jhom_{\dR})$, so that there is a unique $\Z\sym{n}$-module homomorphism $\hJhom_{\dR}\colon S_\Z^\lambda \to M_\Z(\alpha|\beta)$ such that $\hJhom_{\dR} \circ \psi = \Jhom_{\dR}$. The theorem immediately follows.
\end{proof}

Let $\F$ be a field. Since
\begin{gather*}
S_\F^{\lambda} \cong \F \otimes_{\mathbb{Z}} S_\Z^{\lambda} \qquad \text{and} \qquad M_\F(\alpha|\beta) \cong \F \otimes_{\mathbb{Z}} M_\Z(\alpha|\beta),
\end{gather*}
each $\hJhom_{\dR}$ gives rise to an $\F\sym{n}$-homomorphism $\hJhom_{\dR}^{\F} \colon S_\F^{\lambda} \to M_\F(\alpha|\beta)$. More specifically, for any $k \in \mathbb{Z}$, let $k^{\F} =k \cdot 1_{\F} \in \F$. Then
\begin{gather*}\hJhom_{\dR}^{\F} (e_{\t}) = \sum_{d\in\Gamma} \a_{\rho_{\t}^{-1}d,\dR}^{\F} \,(d \otimes \bb{1} \otimes \epsilon).\end{gather*}

As we mentioned in the paragraph following Theorem \ref{T: hom}, our $\hJhom_{\dR}^{\F}$'s generalise the $\hat{\theta}_{\T}$'s constructed by James.

\section{Basis} \label{S:basis?}

Let $\lambda$ be a partition of $n$ and let $(\alpha|\beta)$ be a bicomposition of $n$. As before, fix a $\lambda$-tableau $\t_0$ so that $\sym{n}$ acts on $\TT(\lambda,(\alpha|\beta))$ through $\t_0$. Fix a left transversal $\Gamma$ of $\sym{\alpha|\beta}$ in $\sym{n}$, and write \begin{gather*}\Gamma_{\mathrm{sstd}}=\{ d \in \Gamma \colon \T_d \in \TTS(\lambda, (\alpha|\beta)) \}.\end{gather*} Recall from Corollary \ref{C:sstdRC} that we have $\Gamma_{\mathrm{sstd}}\subseteq \R\cap \C$. It is easy to see that $|\Gamma_{\mathrm{sstd}}|=|\TTS(\lambda,(\alpha|\beta))|$.
In the last section, we constructed for each $\dR \in \R$ an $\F\sym{n}$-homomorphism $\hJhom_{\dR}^{\F} \in \Hom_{\F\sym{n}}\big(S_\F^\lambda,M_\F(\alpha|\beta)\big)$ which generalises James's construction of $\hat{\theta}_{\T} \in \Hom_{\F\sym{n}}\big(S_\F^\lambda,M_\F^{\mu}\big)$ where $\T \in \TT(\lambda,\mu)$.

Let
\begin{alignat*}{3}
 & \hJHOM_{\R}:=\big\{ \hJhom_{\dR} \colon \dR \in \R \big\},\qquad & & \hJHOM^\F_{\R}:=\big\{ \hJhom_{\dR}^\F \colon \dR \in \R \big\}, & \\
 & \hJHOM_{\mathrm{sstd}}:=\big\{ \hJhom_{\dR} \colon \dR \in \Gamma_{\mathrm{sstd}} \big\},\qquad && \hJHOM^\F_{\mathrm{sstd}}:=\big\{ \hJhom_{\dR}^\F \colon \dR \in \Gamma_{\mathrm{sstd}} \big\}.&
\end{alignat*}

In James's classical case, i.e., when $\beta=\varnothing$, the $\hat{\theta}_{\T}$'s, as $\T$ runs over all the semistandard $\lambda$ tableaux of type $\alpha$, form a basis for $\Hom_{\F\sym{n}}\big(S_\F^\lambda,M_\F^\alpha\big)$, unless $\F$ has characteristic $2$ and $\lambda$ is $2$-singular (see \cite[Theorem~13.13]{GJ}).

In this section, we investigate the circumstances under which $\hJHOM_{\mathrm{sstd}}^\F$ is a basis for \linebreak $\Hom_{\F\sym{n}}\big(S_\F^\lambda,M_\F(\alpha|\beta)\big)$.

\subsection{Spanning}

We give two examples in which there is no semistandard $\lambda$-tableau of type $(\alpha|\beta)$, but \linebreak $\Hom_{\F\sym{n}}\big(S_\F^\lambda,M_\F(\alpha|\beta)\big)\ne 0$.

\begin{Example} \label{E:Jameseg}
We have
\begin{align*}
\Hom_{\F\sym{n}}\big(S_{\F}^{\lambda}, M_{\F}(\varnothing|(n))\big)
&= \Hom_{\F\sym{n}}\big(S_{\F}^{\lambda}, \sgn\big)
\cong\Hom_{\F\sym{n}}\big(\sgn, \big(S_{\F}^{\lambda}\big)^*\big) \\
&\cong \Hom_{\F\sym{n}}\big(\sgn, S_{\F}^{\lambda'} \otimes \sgn\big)
\cong \Hom_{\F\sym{n}}\big(\F, S_{\F}^{\lambda'}\big).
\end{align*}
In \cite[Theorem~24.4]{GJ}, a necessary and sufficient condition in terms of the characteristic of $\F$ (as well as the partition $\lambda$) is obtained for $\Hom_{\F\sym{n}}\big(\F, S_{\F}^{\lambda'}\big)$ to be non-zero. As long as $\lambda \ne (1^n)$, there is no semistandard $\lambda$-tableau of type~$(\varnothing|(n))$.
\end{Example}

\begin{Example}\label{Eg: hom char}
Let $\ell$ be an odd prime integer, and let $p:=\Char(\F)$ be either odd or zero.
Let $\lambda = \big(\ell,2,1^{\ell-2}\big)$ and $\alpha = (\ell) = \beta$. Clearly, $\TTS\big(\big(\ell,2,1^{\ell-2}\big),((\ell)|(\ell))\big)=\varnothing$.

Firstly, $M_\F(\alpha|\beta) = \Ind^{\sym{2\ell}}_{\sym{\ell} \times \sym{\ell}} (\F \boxtimes \sgn) \cong \Ind^{\sym{2\ell}}_{\sym{\ell} \times \sym{\ell}} \big(S_\F^{(\ell)} \boxtimes S_\F^{(1^\ell)}\big)$ has a Specht filtration with two factors: $S_\F^{(\ell+1,1^{\ell-1})}$ at the top and $S_\F^{(\ell,1^\ell)}$ at the bottom \cite{JP}, which are simple unless $p = \ell$ \cite[Theorem~2]{P}.
Furthermore, unless $p = \ell$, the partitions $\big(\ell,2,1^{\ell-2}\big)$, $\big(\ell+1,1^{\ell-1}\big)$ and $(\ell,1^\ell)$ have distinct $p$-cores and hence the Specht modules they label lie in distinct blocks of $\sym{2\ell}$ by the Nakayama rule \cite[Corollary~6.1.42]{GJAK}, so that $\Hom_{\F\sym{2\ell}}\big(S_\F^\lambda,M_\F(\alpha|\beta)\big) =0$.

On the other hand, if $p = \ell$, then $S_\F^{(\ell,1^\ell)}$, a submodule of $M_\F(\alpha|\beta)$, is a non-split extension of $D^{(\ell,2,1^{\ell-2})}$ by $D^{(\ell+1,1^{\ell-1})}$. From the known radical structures of Specht modules lying in defect $2$ blocks of symmetric group algebras (see \cite[Theorem~4.4]{R} and \cite[Proposition~6.2]{CT}), we see that~$S_\F^{\lambda}$ has a quotient which is a non-split extension of $D^{(\ell,2,1^{\ell-2})}$ by $D^{(\ell+1,1^{\ell-1})}$ as well. Since all nonzero extensions between simple modules lying in defect~$2$ blocks of symmetric group algebras are one-dimensional \cite[Theorem~I(5)]{S}, we see that $\Hom_{\F\sym{2\ell}}\big(S_\F^\lambda,M_\F(\alpha|\beta)\big) \ne 0$.

Thus, $\Hom_{\F\sym{2\ell}}\big(S_\F^\lambda,M_\F(\alpha|\beta)\big) \ne 0$ if and only if $p = \ell$.
\end{Example}

The two examples above illustrate the difficulty in determining a good sufficient condition in general for $\hJHOM^{\F}_{\mathrm{sstd}}$ to span $\Hom_{\F\sym{n}}\big(S_\F^\lambda,M_\F(\alpha|\beta)\big)$.

\subsection{Linear independence}

While James's $\hat{\theta}_{\T}$ is always nonzero irrespective of the characteristic of the ground field, our $\hJhom_{\dR}$ may be zero in some characteristic in view of Lemma \ref{L:a}(iii). As such, it is certainly possible for $\hJHOM_{\mathrm{sstd}}^{\F}$ to be linearly dependent. In fact, even when all elements of $\hJHOM_{\mathrm{sstd}}^{\F}$ are nonzero, it is still possible for $\hJHOM_{\mathrm{sstd}}^{\F}$ to be linearly dependent:

\begin{Example} \label{E:notli}
We continue with Example \ref{E:a}, where $\lambda = (2,1^{p+2})$ and $(\alpha|\beta) = \big(\varnothing|\big(p,2^2\big)\big)$. Suppose that $p=\Char(\F)$ is an odd prime. Then
\begin{gather*}
-\a^{\F}_{d_1,d_3} = \a^{\F}_{d_1,d_2} = ((p-1)! \cdot 4) \cdot 1_\F \ne 0_{\F}, \\
\a^{\F}_{d_2,d_2} = \a^{\F}_{d_3,d_2} = \a^{\F}_{d_2,d_3} = \a^{\F}_{d_3,d_3} = 0_{\F}.
\end{gather*}
Thus, $\hJhom_{d_2}^{\F} = -\hJhom_{d_3}^{\F}$ by Lemma \ref{L:a}(i).
\end{Example}

In this subsection, we obtain a sufficient condition for $\hJHOM_{\mathrm{sstd}}^{\F}$ to be linearly independent.

We first introduce a pre-order $\unrhd$ on $\TT(\lambda,(\alpha|\beta))$, which induces another on $\sym{n}$.

For each $\T \in \TT(\lambda,(\alpha|\beta))$, write $C_j(\T)$ for the multi-set associated to the $j$th column of $\T$, i.e.,
\begin{gather*}
C_j(\T) = \{ \T(1,j), \T(2,j), \dotsc,\T(r,j) \}\subseteq \{\bb{c}_1,\bb{c}_2,\ldots,\bb{d}_1,\bb{d}_2,\ldots\},\end{gather*} where $r=(\lambda')_j$. Recall that we have the total order \begin{gather*}\bb{c}_1<\bb{c}_2<\cdots<\bb{d}_1<\bb{d}_2<\cdots.\end{gather*}
Let $\T, \T' \in \TT(\lambda,(\alpha|\beta))$. Suppose that $C_j(\T) = \{ y_1,\ldots,y_r\}$ and $C_j(\T') = \{z_1,\ldots,z_r \}$, where $y_1\leq y_2 \leq \dotsb \leq y_r$ and $z_1 \leq z_2 \leq \dotsb \leq z_r$. Write $C_j(\T) \rhd C_j(\T')$ if and only if there exists $k$ such that $y_s = z_s$ for all $1 \leq s < k$ and $y_k>z_k$. We write $\T \unrhd \T'$ if and only if $C_j(\T) = C_j(\T')$ for all $j$, or there exists $t$ such that $C_j(\T) = C_j(\T')$ for all $j < t$ and $C_t(\T) \rhd C_t(\T')$. It is easy to check that $\unrhd$ is a pre-order on $\TT(\lambda,(\alpha|\beta))$ (i.e., it is a reflexive and transitive binary relation on $\TT(\lambda,(\alpha|\beta))$). In addition, write $\T \sim \T'$ if and only if $\T \unrhd \T'$ and $\T' \unrhd \T$ (equivalently, $C_j(\T) = C_j(\T')$ for all $j$), and $\T \rhd \T'$ if and only if $\T \unrhd \T'$ but $\T' \ntrianglerighteq \T$ (equivalently, there exists~$t$ such that $C_j(\T) = C_j(\T')$ for all $j < t$ and $C_t(\T) \rhd C_t(\T')$).

Recall that $\T_d = d \cdot \T_0$. Let $d,d' \in \sym{n}$, and write $d \unrhd d'$, $d \sim d'$ and $d \rhd d'$ if and only if $\T_d \unrhd \T_{d'}$, $\T_d \sim \T_{d'}$ and $\T_d \rhd \T_{d'}$ respectively.

\begin{Lemma} \label{L:pre-order}
	Let $\lambda$ be a partition of $n$, $(\alpha|\beta)$ be a bicomposition of $n$ and $d,d' \in \sym{n}$.
	\begin{enumerate}\itemsep=0pt
\item[$(i)$] $d \sim d'$ if and only if $d' \in C_{\t_0} d \sym{\alpha|\beta}$.
\item[$(ii)$] If $\T_d$ is row semistandard, and $d' \in R_{\t_0} d \sym{\alpha|\beta}$, then either $\T_{d'} = \T_d$ or $\T_{d'} \rhd \T_d$.
\item[$(iii)$] If $\T_{d} \in \TTS(\lambda, (\alpha|\beta))$ and $\a_{d', d} \ne 0$, then $d' \unrhd d$.
	\end{enumerate}
\end{Lemma}

\begin{proof}
	Parts (i) and (ii) are straightforward. For part (iii), if $\a_{d',d} \ne 0$, then there exist $\sigma \in C_{\t_0}$, $\tau \in R_{\t_0}$ and $\xi \in \sym{\alpha|\beta}$ such that $\sigma d' = \tau d \xi$. Thus,
	\begin{gather*}d' \sim \sigma d' = \tau d \xi \unrhd d\end{gather*} by parts (i) and (ii) since $\T_d$ is, in particular, row semistandard.
\end{proof}

We can now state a sufficient condition for $\hJHOM_{\mathrm{sstd}}^\F$ to be linearly independent.

\begin{Proposition} \label{P:phihatli}
	Suppose that $\a^{\F}_{\dR,\dR} \ne 0$ for all $\dR \in \Gamma_{\mathrm{sstd}}$. Then $\hJHOM_{\mathrm{sstd}}^{\F}$ is linearly independent.
\end{Proposition}

\begin{proof}
	Suppose the contrary that $\hJHOM_{\mathrm{sstd}}^{\F}$ is linearly dependent, say $\sum\limits_{i=1}^r c_i \hJhom^{\F}_{\dR_i} = 0$, where, for each $1 \leq i \leq r$, $\dR_i \in \Gamma_{\mathrm{sstd}}$ and $c_i \in \F \setminus \{0\}$. Since the $\T_{\dR_i}$'s are all column semistandard and distinct, we see that $\dR_i \not\sim \dR_j$ for all $i \ne j$. Thus, relabelling if necessary, we may assume that $\dR_1 \ntrianglerighteq \dR_i$ for all $i \geq 2$. Then
	\begin{gather*}
	0 = \left (\sum _{i=1}^r c_i \hJhom^{\F}_{\dR_i}\right )(e_{\t_0} ) =
	\sum_{i=1}^r c_i \sum_{d \in \Gamma} \a^{\F}_{d,\dR_i} (d \otimes \bb{1} \otimes \epsilon) =
	\sum_{d \in \Gamma} \left (\sum_{i=1}^r c_i \a^{\F}_{d,\dR_i}\right ) (d \otimes \bb{1} \otimes \epsilon),
	\end{gather*}
	so that $\sum\limits_{i=1}^r c_i \a^{\F}_{d,\dR_i} = 0$ for all $d \in \Gamma$. In particular, $\sum\limits_{i=1}^r c_i \a^{\F}_{\dR_1,\dR_i} = 0$. But for each $2 \leq i \leq r$, we have $\dR_1 \ntrianglerighteq \dR_i$, so that $\a_{\dR_1,\dR_i} = 0$ by Lemma \ref{L:pre-order}(iii), and hence $\a^{\F}_{\dR_1,\dR_i} = 0$. Consequently, $c_1 \a^{\F}_{\dR_1,\dR_1} = 0$, contradicting $c_1 \ne 0$ and $\a^{\F}_{\dR_1,\dR_1} \ne 0$.
\end{proof}

The next lemma gives us some idea of what $\a_{\dR,\dR}^{\F}$ is for $\dR \in \Gamma_{\mathrm{sstd}}$.

\begin{Lemma} \label{L:a_[d,d]}
Let $\dR \in \R$ such that $\T_\dR \in \TTS(\lambda,(\alpha|\beta))$. Then \begin{gather*}\a_{\dR,\dR} = |\stab_{C_{\t_0}}(\T_\dR)|.\end{gather*}
\end{Lemma}

\begin{proof}	Observe that $C_{\t_0} \dR \sym{\alpha|\beta} \cap R_{\t_0} \dR \sym{\alpha|\beta} = \dR \sym{\alpha|\beta}$: if $\sigma \dR = \tau \dR \xi$, where $\sigma \in C_{\t_0}$, $\tau \in R_{\t_0}$ and \smash{$\xi \in \sym{\alpha|\beta}$}, then $\dR \sim \tau \dR$ by Lemma~\ref{L:pre-order}(i) so that $\T_{\tau \dR}\ntriangleright \T_\dR$, while $\tau \dR \sim \dR$ only if $\tau \dR \in \dR \sym{\alpha|\beta}$ by Lemma~\ref{L:pre-order}(ii). Applying Lemma~\ref{L:a}(iii) completes the proof since $\stab_{C_{\t_0}}(\T_\dR)=C_{\t_0}\cap \dR\sym{\alpha|\beta}\dR^{-1}$.
\end{proof}

Combining the last two results, we get the following immediate corollary.

\begin{Corollary} \label{C:li}
If $\Char(\F) \nmid |\stab_{C_{\t_0}}(\T_\dR)|$ for all $\dR \in \Gamma_{\mathrm{sstd}}$, then $\hJHOM_{\mathrm{sstd}}^{\F}$ is linearly independent.
\end{Corollary}

\subsection{Main results}

Similar to how our $\hJhom_{\dR}$ induces $\hJhom_{\dR}^{\F}$, every $\phi \in \Hom_{\Z\sym{n}}\big(S_{\Z}^{\lambda}, M_{\Z}(\alpha|\beta)\big)$ also induces \linebreak $\phi^{\F} \in \Hom_{\F\sym{n}}\big(S_{\F}^{\lambda}, M_{\F}(\alpha|\beta)\big)$. Write
\begin{gather*}
\Phi^{\F} := \big\{ \phi^{\F} \colon \phi \in \Hom_{\Z\sym{n}}\big(S_{\Z}^{\lambda}, M_{\Z}(\alpha|\beta)\big)\big\}.
\end{gather*}
Then $\Phi^{\F}$ is a $\F$-subspace of $\Hom_{\F\sym{n}}\big(S_{\F}^{\lambda}, M_{\F}(\alpha|\beta)\big)$ containing $\hJHOM^{\F}_{\R}$~-- and hence $\hJHOM^{\F}_{\mathrm{sstd}}$~-- as a~subset.

\begin{Proposition} \label{P:LIimpliesbasis}
	Let $\Psi^{\F}$ be a linearly independent subset of $\Phi^{\F}$ of size $|\TTS(\lambda, (\alpha|\beta)|$. Then
\begin{enumerate}\itemsep=0pt
 \item [$(i)$] $\Psi^{\F}$ is a basis for $\Phi^{\F}$, and
 \item [$(ii)$] if $S_{\F}^{\lambda}$ lies in a block of $\F\sym{n}$ which is simple as an algebra, then \begin{gather*}\Hom_{\F\sym{n}}\big(S^{\lambda}_{\F}, M_\F(\alpha|\beta)\big)=\Phi^\F.\end{gather*}
\end{enumerate}
\end{Proposition}

\begin{proof}
	If $S_\F^{\lambda}$ lies in a block of $\F\sym{n}$ which is simple as an algebra, then the composition multiplicity of $S_\F^{\lambda}$ in $M_\F(\alpha|\beta)$ equals $\dim_{\F} \Hom _{\F \sym{n}}\big(S_\F^{\lambda}, M_\F(\alpha|\beta)\big)$, which is in turn equal to $|\TTS(\lambda, (\alpha|\beta)|$ by Theorem~\ref{T:Specht factors of sgn permutation}. Since $\Psi^{\F}$ is linearly independent with the same cardinality, it is a basis for $\Hom _{\F \sym{n}}(S_\F^{\lambda}, M_\F(\alpha|\beta))$. In particular, since \begin{gather*}\Hom _{\F \sym{n}}\big(S_\F^{\lambda}, M_\F(\alpha|\beta)\big) = \F\text{-span}\big(\Psi^{\F}\big) \subseteq \Phi^{\F} \subseteq \Hom _{\F \sym{n}}\big(S_\F^{\lambda}, M_\F(\alpha|\beta)\big),\end{gather*} we must have equality throughout, proving part (ii).
	
	For part (i), let $\Psi^{\F} = \big\{ \phi_1^{\F}, \dotsc, \phi_k^{\F} \big\}$, where $k = |\TTS(\lambda, (\alpha|\beta)|$, and let $\phi_0^{\F} \in \Phi^{\F}$. Since $\mathbb{Q}\sym{n}$ is~semisimple as an algebra, we apply the previous paragraph and obtain that \linebreak $\Hom_{\mathbb{Q}\sym{n}}(S_{\mathbb{Q}}^{\lambda}, M_{\mathbb{Q}}(\alpha|\beta))$ has dimension $k$. Thus there is a non-trivial relation on $\phi_0^{\mathbb{Q}}, \phi_{1}^{\mathbb{Q}}, \dotsc, \phi_{k}^{\mathbb{Q}}$, which we can write as $\sum\limits_{j =0}^k c_j \phi_{j}^{\mathbb{Q}} = 0$, where the $c_j$'s are coprime integers. This yields the non-trivial linear relation $\sum\limits_{j =0}^k c_j^{\F} \phi_{j}^{\F} = 0$. Since $\Psi^{\F}$ is linearly independent, this implies that $\phi_{0}^{\F}$ lies in the $\F$-span of~$\Psi^{\F}$.
\end{proof}

\begin{Remark}
As is well-known, $S^{\lambda}_{\F}$ lies in a simple block if and only if $\lambda$ is a {\em $p$-core partition} (in other words, $\lambda$ has no rimhook of size $p$) where $p=\Char(\F)$.
\end{Remark}

As an immediate corollary, we have

\begin{Corollary}The dimension of $\Phi^{\F}$ is at most $|\TTS(\lambda,(\alpha|\beta))|$.
\end{Corollary}

The following is our first main result:

\begin{Theorem} \label{T:ssbasis}Suppose that $\F\sym{n}$ is semisimple as an algebra $($or equivalently, $\Char(\F) = 0$ or $\Char(\F)>n)$. Then $\hJHOM_{\mathrm{sstd}}^{\F}$ is a basis for $\Hom _{\F \sym{n}}\big(S_\F^{\lambda}, M_\F(\alpha|\beta)\big)$.
\end{Theorem}

\begin{proof}Since $\F \sym{n}$ is semisimple, every block of $\F\sym{n}$ is simple. Thus, $\hJHOM_{\mathrm{sstd}}^{\F}$ is a basis for $\Hom _{\F \sym{n}}\big(S_\F^{\lambda}, M_\F(\alpha|\beta)\big)$ by Corollary~\ref{C:li} and Proposition~\ref{P:LIimpliesbasis}.
\end{proof}

\begin{Corollary}
Let $\phi \in \Hom_{\Z\sym{n}}\big(S^{\lambda}_{\Z}, M_{\Z}(\alpha|\beta)\big)$. Then $\phi$ lies in the $\mathbb{Q}$-span of $\hJHOM_{\mathrm{sstd}}$.
\end{Corollary}

\begin{proof}
By Theorem \ref{T:ssbasis}, $\hJHOM_{\mathrm{sstd}}^{\mathbb{Q}}$ is a basis for $\Hom _{\mathbb{Q} \sym{n}}\big(S_\Q^{\lambda}, M_\Q(\alpha|\beta)\big)$, so that $\phi^{\Q} = \sum\limits_{\dR \in \Gamma_{\mathrm{sstd}}} c_{\dR} \hJhom^{\Q}_{\dR}$, where $c_{\dR} \in \mathbb{Q}$ for all $\dR \in \Gamma_{\mathrm{sstd}}$. But this implies that $\phi = \sum\limits_{\dR \in \Gamma_{\mathrm{sstd}}} c_{\dR} \hJhom_{\dR}$.
\end{proof}

Our next main result provides a sufficient condition for $\hJHOM_{\mathrm{sstd}}^{\F}$ to be a basis for $\Phi^{\F}$ when $\F\sym{n}$ is not semisimple. Note that Example~\ref{Eg: hom char} shows that this condition is insufficient in ensuring that $\hJHOM^{\F}_{\mathrm{sstd}}$ spans $\Hom_{\F\sym{n}}\big(S_{\F}^{\lambda}, M_{\F}(\alpha|\beta)\big)$.

\begin{Theorem} \label{T:3}
Suppose that $\Char(\F) = p >0$. If no column of $\T$ has $p$ or more nodes of the same colour for every $\T \in \TTS(\lambda, (\alpha|\beta))$, then $\hJHOM_{\mathrm{sstd}}^{\F}$ is a basis for the $\Phi^{\F}$.
\end{Theorem}

\begin{proof}Let $\dR \in \Gamma_{\mathrm{sstd}}$. Observe that $p \nmid |\mathrm{stab}_{C_{\t_0}}(\T_\dR)|$, if and only if no column of $\T_\dR$ has $p$ or more nodes of the same colour. Thus, this follows from Corollary~\ref{C:li} and Proposition~\ref{P:LIimpliesbasis} immediately.
\end{proof}

In particular:

\begin{Corollary}
 Suppose that $\ell(\lambda)< \Char(\F)$. Then $\hJHOM_{\mathrm{sstd}}^{\F}$ is a basis for $\Phi^\F$.
\end{Corollary}
	
We end our paper with the following remark and leave the details to the reader.

\begin{Remark} Recall from Definition \ref{D:RC} that the set $\R$ and hence the sets $\JHOM_\R$ and $\JHOM_{\mathrm{sstd}}$ depend on the fixed $\lambda$-tableau $\t_0$ and $(\alpha|\beta)$ we chose. Suppose that we had chosen another $\lambda$-tableau $\t_0'$ and arrived at the sets $\R'$, $\JHOM_{\R'}$ and $\JHOM'{}_{\mathrm{sstd}}$. It is easy to verify that
\begin{gather*}
\R'=\pi\R,\qquad \JHOM_{\R'}=\JHOM_\R,\qquad \JHOM'_{\mathrm{sstd}}= \JHOM_{\mathrm{sstd}},
\end{gather*} where $\pi=\t_0'\circ (\t_0)^{-1} \in \sym{n}$.

On the other hand, let $\{1,2,\dotsc,n\} = \bigcup\limits_{i=1}^r A_i \cup \bigcup\limits_{j=1}^s B_j$ (disjoint union throughout), with $A_i,B_j \ne \varnothing$ for all $i$ and $j$. The set $\R$ in Definition \ref{D:RC} and the homomorphisms $\Jhom_{\dR}$ in Section~\ref{S:construction} can be generalised to yield the set $\R_{A|B}$ and homomorphisms $\Jhom_{\dR}^{A|B} \in\Hom_{\Z\sym{n}}\big(S^\lambda_\Z,M_\Z(A|B)\big)$, where
\begin{gather*}
M_\Z(A|B) = \Ind_{\sym{A|B}}^{\sym{n}} \big(\Z_{\sym{A}} \boxtimes \sgn_{\sym{B}}\big) ,\\
\sym{A} = \prod_{i=1}^r \sym{A_i}, \qquad \sym{B} = \prod_{j=1}^s \sym{B_j}, \qquad \sym{A|B} = \sym{A}\sym{B},
\end{gather*}
giving us a subset $\JHOM^{A|B}_{\R_{A|B}}$ of $\Hom_{\Z\sym{n}}\big(S^\lambda_\Z,M_\Z(A|B)\big)$.

Let $g\in\sym{n}$, and let $A_i'=g^{-1}(A_i)$, $B_j'=g^{-1}(B_j)$ for all $i=1,\ldots,r$ and $j=1,\ldots,s$. Then $\{1,2,\dotsc,n\} = \bigcup\limits_{i=1}^r A'_i \cup \bigcup\limits_{j=1}^s B'_j$ (disjoint union throughout), and it is straightforward to show that
\begin{gather*}
 \R_{A'|B'}=\R_{A|B}g,\qquad \JHOM^{A'|B'}_{\R_{A'|B'}} = \varrho_g \circ \JHOM^{A|B}_{\R_{A|B}},
\end{gather*} where $\varrho_g$ is the natural $\Z\sym{n}$-module isomorphism $M_\Z(A|B) \to M_\Z(A'|B')$ defined by $\varrho_g(x \otimes \bb{1} \otimes \epsilon) = xg \otimes \bb{1} \otimes \epsilon$ for all $x \in \sym{n}$.
\end{Remark}

\subsection*{Acknowledgements}

Supported by Singapore MOE Tier 2 AcRF MOE2015-T2-2-003. We thank the guest editors for bringing to our attention the work of Du and Rui on signed $q$-permutation modules of the Iwahori--Hecke algebras of type~$A$~\cite{DR}, and the referees for their helpful comments.

\pdfbookmark[1]{References}{ref}
\LastPageEnding

\end{document}